\documentclass[12pt,twoside]{article}

\usepackage{setspace}
\usepackage{amssymb}
\usepackage{amssymb}
\usepackage{bbm}
\usepackage{comment}
 \usepackage{amsthm}
\usepackage{amsmath}
\usepackage{graphicx}
\usepackage{color}
\usepackage{amsmath}
\usepackage{dutchcal}
\usepackage{amssymb}
\usepackage{amsthm}
\usepackage{graphicx,color}
\usepackage{pstricks}
\usepackage{pst-tree}
\usepackage{multido}
\usepackage{comment}
\def\R{{\mathbb R}}
\def\dd{{\mathrm d}}
\def\N{{\mathbb N}}

\def\cB{{\mathcal B}}
\def\cH{{\mathcal H}}
\def\cF{{\mathcal F}}

\def\dd{{\rm d}}

\newtheorem{thm}{Theora}[section]
\newtheorem{theo}[thm]{Theorem}

\newtheorem{cor}[thm]{Corollary}

\newtheorem{prop}[thm]{Proposition}

\newtheorem{rem}[thm]{Remark}
\newtheorem{example}[thm]{Example}
\newcommand\dint{\displaystyle\int}
\newcommand\ds{\displaystyle\sum}

\newcommand*{\MyChi}{\raisebox{0.35ex}{\( \chi \)}}%

\begin{document}

 \title{A class of L\'evy  driven SDEs and their explicit invariant measures}
  \def\lhead{S.\ ALBEVERIO, B. Smii, L.\ Di Persio} 
  \def\rhead{Invariant measures for stochastic differential equations driven by L\'evy noise} 


\author{Sergio Albeverio\footnote{Dept. Appl. Mathematics, University of Bonn, HCM, BiBoS, IZKS; CERFIM, Locarno. {\tt albeverio@uni-bonn.de}}; %
  \\
  Luca Di Persio\footnote{University of Verona, Department of Computer Science, strada Le Grazie, 15, Verona,
    Italia. {\tt luca.dipersio@univr.it}}
    \\
  Elisa Mastrogiacomo\footnote{Universit\'a degli Studi di Milano Bicocca,
   Dipartimento di Statistica e Metodi Quantitativi,
   Piazza Ateneo Nuovo, 1 20126 Milano
    Italia. {\tt elisa.mastrogiacomo@unimib.it}}
   \\
 Boubaker Smii\footnote{King Fahd University of
Petroleum and Minerals, Dept. Math. and Stat., Dhahran 31261, Saudi
Arabia. \hspace*{5mm}{\tt boubaker@kfupm.edu.sa}}}
\date{}
\maketitle

\begin{abstract}

We describe a class of explicit invariant measures for both finite and infinite dimensional Stochastic Differential Equations (SDE) driven by L\'evy noise.
We first discuss in details the finite dimensional case with a linear, resp. non linear, drift. In particular, we exhibit a class of such SDEs for which the invariant measures are given in explicit form, coherently in all dimensions. We then indicate how to relate them to invariant measures for SDEs on separable Hilbert spaces.
\end{abstract}
\newpage
\tableofcontents
\newpage

\section{Introduction}\label{Introduction}
\subsection{Motivations and contents} \label{Motivations}
In the study of phenomena described by evolution equations and
stochastic processes the use of invariant measures plays an
important role, both from a theoretical and an applied point of
view. This is due to the fact that the presence of invariant
measures permits, in particular, to have a grip on the asymptotic
behaviour in time of the processes involved and often (in the
presence of ergodicity) to compute time averages of functionals, at
least, approximately, by averaging with respect to the invariant
measure.This is at the very basis of statistical mechanics, where
the invariant measure is the Gibbs measure, see, e.g.,
\cite{AKKR,Ru,SWEL}. The same idea has also been used in connection
with continuum systems, e.g. in hydrodynamics, see, e.g.,
\cite{AlbCru, AFER, AlbFHK, AFS}, and quantum field, see, e.g.,
\cite{AGoWu, AlGYo, AHK1, AHKS, AKR, JoMi, S.Mitter, Pa, ParWu, Si}.
Also in the general theory of dynamical system, invariant measures
play an important role. According to a principle of Kolmogorov the
finding of invariant measures for such systems might be facilitated
by perturbing slightly and stocastically the system, see \cite{Elg}
Invariant measures have also been intensively discussed in
connection with stochastic partial differential equations (SPDEs)
and, more generally, with stochastic processes, where they are the
basis of all Monte-Carlo methods, see, e.g., \cite{Pa, ParWu}. For
both theoretical and practical reasons it is useful to have
expressions for invariant measures which are as explicit as
possible. Often they also have invariance properties with respect to
transformations in state space, which makes them particularly
useful, reflecting important symmetry properties of the underlying
systems.
\newline
This paper is devoted to the search of such explicit measures for
(in)finite dimensional SDE driven by L\'evy noise and with nonlinear
drift coefficients. This connects to our previous paper
\cite{ADPMS1}, where we studied such equations in the infinite
dimensional case. In that paper we found, in particular, {\it
abstract} invariant probability measures for the equations at hand
and we discussed their relations with a decomposition
 of the solution process as a sum of a stationary component and an asymptotically in time vanishing component.
In the present work we reconsider the question of invariant measures
having in mind to characterize them explicitly, at least in some
cases we discuss particularly the case where the driving noise contains a jump component, since the case of driving noise of pure Gaussian type was already discussed,
for our system, in \cite{AlDiPMa}.

\noindent
In section $1.2$ we summarize basic concepts of the theory of Markov
semigroups, generators and Dirichlet forms, since they are basic for
the rest of the paper.

\noindent
 In chapter \ref{Section2} of the present
paper we concentrate ourselves on the finite dimensional case. This
serves as a basis for going over to the infinite dimensional case,
in the subsequent chapter $3$.

\noindent
 In Section $2.1$ we recall results
related to the case of linear drifts, i.e., for
Ornstein-Uhlenbeck-L\'evy (OUL) processes, where a complete
classification of invariant measures has been obtained,
particularly by work of Sato and Yamazato, see \cite{APD,Sat91,Sato,SY84,Yama}.\\
In Section \ref{Section2.2} we discuss invariant measures for
OUL-processes perturbed by nonlinear drifts, following and extending
basically work of \cite{ABRW} and \cite{BeSc}. We give here some of
the details since the methods are also useful for the later section
$2.4$.\\
In Section $2.3$ we discuss the symbol associated with solutions of
SDE, stressing the explicit form of the associated generators,
having in mind
concrete applications in Section $2.4$.\\
In Section $2.4$ we start from explicit invariant measures and
construct associated L\'evy-type generators and SDE. This is related
to techniques known in the case of Gaussian noise as Dynkin's {\it
h-transform} or, {\it ground state transformation}, see.
\cite{AHKS}.The extension to the L\'evy case was initiated by
\cite{A-Cufaro}, we give some observations and complements to this
construction, stressing both its relation to the symbols discussed
in Section \ref{Section2.3} and the invariant measures. The
discussion is then extended in Section \ref{Section2.5} considering
perturbed O-U-L\'evy processes, defined by invariant measures and
Dirichlet forms. In chapter $3$ we discuss  the infinite dimensional case.\\
Section $3.1$ presents the case of an infinite dimensional O-U
L\'{e}vy-process, following
basic work by \cite{ChoMi}, stressing also the relation with our paper \cite{ADPMS1}.\\
Section  $3.2$  presents the case of certain infinite dimensional
L\'evy driven systems , which can be seen as infinite dimensional limits
of the finite
dimensional systems discussed in Section \ref{Section2.4}.\\

\subsection{Basic concepts on Markov semigroups, generators, Dirichlet forms.}

A transition function on a Polish space $\mathcal{E}, \mathcal{B}(\mathcal{E})$, e.g.\ $\mathbb{R}^d, \mathcal{B}(\mathbb{R}^d)$, is
by definition a family of mappings $p_{s,t}(x,B)$, $x \in \mathcal{E}$,
$B \in \cB(\mathcal{E})$, with $0 \leq s \leq t < \infty$, and with values in $[0,1]$ with the properties:
\begin{enumerate}
\item $p_{s,t}(x, B)$ it is a probability measure as a function of $B$ for any fixed $x$;
\item it is measurable in $x$ for any fixed $B$;
\item $p_{s,s}(x, B) = \delta_x(B)$ for $s \geq 0$;
\item \label{it:ChKo}it satisfies
\begin{align*}
 \int_{\mathcal{E}} p_{s,t} (x, \dd y) p_{t,u} (y, B) = p_{s,u}(x,B), \qquad  for\  0\leq  s \leq t \leq u \;,
\end{align*}
 which is called the Chapman-Kolmogorov property.

\end{enumerate}
If, in addition,
\begin{enumerate}
\item[5.] $p_{s+h,t+h}(x, B)$ does not depend on $h$,
\end{enumerate}
then it is called a (temporally homogeneous) transition function and
it is easy to show that it is given by a one-parameter family of {\it Markov kernels} $p_t(x, B), t \geq 0 $, satisfying $1-4$, and such that $p_t (x, B) = p_{s,s+t}(x , B)$
for $s\geq 0$.

 In the case of a (temporally homogeneous)
transition function $4$ is written as
\begin{align*}
 \int_{\mathcal{E}} p_{s} (x, \dd y) p_{t} (y, B) = p_{s+t}(x,B), \qquad  for\   s, t \geq 0.
\end{align*}
This is called the semigroup property of $p_t$, $t \geq 0$.
A probability measure on $\mathcal{E}$ (or, more generally, a measure for which $p_t(X,\cdot)$ is integrable) is said to be invariant under
$p_t$, $t\geq0$, if
$\int_{B}p_t(x,B)\mu(\mathrm{d}x)=
\mu(B)$ for all $t>0$, and all Borel subsets $B$ of $\mathcal{E}$.

Let us also note that a transition function $p_t$ also defines a
semigroup acting on positive measurable functions $f$ on
$\mathcal{E}$,
by $p_t f(x) = \int_\mathcal{E} p_t(x,dy) f(y)$, $x \in \mathcal{E}$. Note that $f = \MyChi_B$, for any Borel subset $B$ of $\mathcal{E}$, we have
$\left( p_t \MyChi_B \right) (x) = p_t(x,B)$. Moreover the semigroup property of $p_t$ implies that $p_t \circ p_s = p_s \circ p_t = p_{t+s}$, for any $s,t \geq 0$.


One extends by linearity $p_t$ to the Banach space
$B(\mathcal{E})$ (complete, normed, linear space) of all the bounded
measurable real (or complex) valued functions $f \in
B(\mathcal{E})$, with norm $\| f\|_u:=\sup\mid f \mid$. From $\mid
p_t (f) \mid \leq \int f(y) p_t(x,dy) \leq  \| f \|_u$ we have that
$p_t$ is contractive , in fact $p_t, t\geq 0$ constitutes a bounded
linear strongly continuous semigroup acting on $B(\mathcal{E})$.
Note that $p_0 f =f \;,\; f \in B(\mathcal{E})$.


A stochastic process $X= (X_t), t \geq 0)$ on a probability space $\left( \mathcal{E}, \mathcal{B}(\mathcal{E}), \mathbb{P} \right)$, is said to be a Markov process with respect to a filtration $\left( \mathcal{F} \right)_{t \geq 0}$ of subsets of $\mathcal{E}$ if
$
 \mathbb{E} \left( X_t \mid \sigma(X_s) ) \right) = \mathbb{E} \left( X_t \mid \mathcal{F}_s\right) \;,\; \forall s \in [0,t]\;,
$
where $\sigma(X_s)$ indicates the $\sigma-algebra$ generated by $X_s$.
For other characterizations of the Markov property see, e.g., \cite{Bau}.
To a Markov family of processes on $\left( \mathcal{E}, \mathcal{B}(\mathcal{E}) \right)$ with probability measure $\mathbb{P}^x$ such that $x \mapsto \mathbb{P}^x\left( X_t \in B\right)$ is measurable for any $ B \in \mathcal{B}(\mathcal{E})$, there is naturally associated a transition function defined by
$$
 p_t^X \left(x,B\right) := \mathbb{P} \left( X_t \in B \mid X_0=x\right) \;,\; B \in \mathcal{B}(\mathcal{E})\;.
$$


By the properties characterizing the transition function we have $p_t f \geq 0$, for $f \geq 0$,  and if $f \leq 1$ then $p_t f \leq 1$, as well as $p_t 1 =1$, where 1 is the function identically equal to $1$ on $\mathcal{E}$. $p_t 1=1$ is sometimes called conservativeness property of $p_t$.



If $p_t^X$ is the transition function of a Markov process $X_t$ on
$\mathcal{B}$ then one shows that $\mu$ is invariant under $p_t^X$ iff
$\mu$ is the initial distribution of $X_t$ and $P(X_t\in B)=\mu(B)$,
for all $t\geq0, B \in \mathcal{B}(\mathcal{E})$. In fact from $ \int p_t(x,B)\mu(dx) =\mu(B)$ one
deduces, by the Markov property and $\mu(B)={\cal L}(X_0),$ that
$P(X_t \in B)=\int p_t(x,B)\mu(dx) =\mu(B)$. Viceversa, if this holds,
by the Markov property we have $ \int_B p_t(x,B)\mu(dx) =\mu(B)$,
hence that $\mu$ is invariant.\\
 This also coincides with the definition of $\mu$ invariant under $P_t$, in the sense that $\int P_t f d\mu= \int f d\mu, $ for all $f \in L^2(\mathcal{E},\mu),$ where
$(P_t)_{t\geq 0}$ is the Markov semigroup associated with
$(x,B)\mapsto p_t(x,\mathcal{B})$ in $L^2(\mathcal{E},\mu)$.

A probability measure $\nu$ on $\mathcal{E}$ is said to be the limit
distribution of a temporally homogeneous Markov process on
$\mathcal{E}$ with transition function $p_t$, $t\geq0$ on
$\mathcal{E}$ if $\lim p_t(x,\cdot)\to\nu$ as $t\to +\infty$, for
any $x\in\mathcal{E}$, in the sense of weak convergence of measures
on $\mathcal{E}$ (i.e.\ in the sense of integrals against functions
in $C_b(\mathcal{E})$). \\
The above definitions are adapted from, e.g.,  \cite[Chapt.3, Sec.17]{Sat91}. \\

To a given transition function $p_t(x,dy)$  there is associated a Markov
process $X_t$ densely defined on a probability space $\left(
\mathcal{E}, \mathcal{B}(\mathcal{E}), \mathbb{P}^x \right)$ such
that $\mathbb{P}^x (X_t \in A) = p_t (x,A)$, for all $A \in
\mathcal{B}(\mathcal{E})$, $x \in \mathcal{E}$. If $\mathcal{E}$ is
a linear space and if $p_t$ is space translation invariant, in the
sense that $p_t(x+a,A) = p_t(x,A+a)$ for all $ a \in \mathcal{E}$,
then $p_t(x,A) = \tilde{p}_t(A-x)$ for some $\tilde{p}_t$, $0 \leq
\tilde{p}_t \leq 1$, $\tilde{p}_t$ a convolution semigroup acting in
$\mathcal{E}$.


In general a strongly continuous semigroup on a Banach space $B$ is a family of bounded maps $T_t,$ $ t \geq0$ on $B$ such that $T_t T_s = T_{t+s}$, $T_0 =1$, $t \mapsto T_t x$ is continuous for every element $x \in B$. Often such semigroups are called  $C_0-semigroups$.
Such semigroups satisfy $\| T_t \| \leq M e^{\omega t}$, for $ t \in [0, +\infty)$, for some constants $M \geq 1$ and $\omega \geq0$.
 One shows, see, e.g., \cite{PaZy}, that given such a semigroup one can associate to it its infinitesimal generator $A$, which turns out to be a linear operator on the dense subset
  $D(A)$ of ${B}$ defined by $D(A) = \left\{ x \in B \mid \lim_{t \downarrow 0} \frac{T_t x -x}{t}= Ax \right\}$, the limit being understood in the norm of ${B}$
  (strong convergence of $\frac{T_t x -x}{t}$ to $Ax$ as $t \downarrow 0$).
Let ${B}^\star$ the dual of a Banach space ${B}$ over $\mathbb{R}$
or $\mathbb{C}$ (i.e. the space of all continuous linear maps from
${B}$ into $\mathbb{R}$ or $\mathbb{C}$). We denote the duality by
$\langle x^\star , x \rangle$, $x^\star \in {B}^\star$, $x \in {B}$.
A linear operator $A$ on ${B}$ is said to be dissipative if for
every $x \in D(A)$ there exists a $x^\star \in F(x) := \left\{ y \in
{B}^\star \mid \langle y,x \rangle = \| x\|_{{B}}^2 =
\|y\|_{{B}^\star}^2 \right\}$, such that $Re\langle Ax, x^\star
\rangle \leq 0 $.

Exploiting the Hahn-Banach theorem we have that $F(x)$ is non empty,
moreover, if ${B}$ is reflexive, then $F(x)$ is composed by a single
element. In the case where $B$ is a Hilbert space, this element
 can be identified with $x$ itself by the canonical duality between a Hilbert space and its dual.
As part of a theorem by Lumer and Phillips, if in addition $A$ is
such that the range of $\mathbbm{1}-A$ is ${B}$ (in which case one
calls $A$ maximal dissipative) then $A$ is the infinitesimal
generator of a $C_0$-semigroup on ${B}$, see, e.g.,
\cite[Th.4.3]{PaZy}. If again in particular ${B}$ is a Hilbert space and
$(-A)$ is positive (i.e. $\langle -A x,x\rangle \geq 0$, for all $x
\in D(A)$ ), with $\langle\cdot,\cdot\rangle$ the scalar product in
${B}$, then $A$ is dissipative. If ${B}$ is complex Hilbert space,
then $-A$ positive is symmetric, as seen from the polarization
formula, see, e.g., \cite{Procesi}.

  We recall that a densely defined operator $T$ in a Hilbert space $\mathcal{H}$ is said to be symmetric if its adjoint $T^\star$ is an extension of $T$ in the sense that  the domain of
  $T^\star$ contains the domain of $T$ and the restriction of $T^\star$ to the domain of $T$ coincides with $T$, and we write $ T \subseteq T^\star$.  Moreover $T$ it is called self-adjoint if $D(T)=D(T^\star)$. A positive self-adjoint operator $A$ in a Hilbert space generates a $C_0-$contraction semigroup $e^{-tA}$, $t \geq 0$.
Viceversa, if a $C_0-$contraction semigroup $T_t$ on a Hilbert space is self-adjoint ( i.e. $T_t^\star = T_t$) then its generator is symmetric and positive.
A special case is the one where $\mathcal{H}$ is a real $L^2-$space , say $L_\mathbb{R}^2 \left(\mathcal{E}, \mu\right)$, with $\mu$ a probability measure on $\left(\mathcal{E},\mathcal{B}(\mathcal{E}))\right)$.
In this case it is natural to consider $C_0-$contraction semigroups $T_t$, $t \geq 0$, which are also symmetric in $L_{\mathbb{R}}^2 \left(\mathcal{E}, \mu\right)$ (hence also self-adjoint being bounded), and in addition are sub-Markov semigroups in the sense that for any $ f \in L_{\mathbb{R}}^2 \left(\mathcal{E}, \mu\right)$, where $0 \leq f \leq 1$, $\mu-a.e.$, one has $ 0 \leq T_t f \leq 1$.

If , in addition, $T_t 1 = 1$, $t\geq 0$, $\mu_a.e.$, i.e. $T_t$ is conservative, then $T_t$ is said to be a Markov semigroup.

It is known that $T_t$ has then a kernel $p_t(x,dy)$ such that
$\left( T_t f \right) (x) = \int f(y) p_t(x,dy)$, $f \in
\mathcal{B}(\mathcal{E})$. One defines namely $p_t(x,B) = T_t
\MyChi_B(x)$, for all $B \in \mathcal{B}(\mathcal{E})$. It follows that $B \rightarrow p_t(x,B)$ is a probability  measure on $\mathcal{B}(\mathcal{E})$. Then $T_t$ coincides on $\mathcal{B}(\mathcal{E})
\subset L^2_\mathbb{R}(\mathcal{E},\mu)$ with the Markov semigroups
$p_t$ given by the kernels $p_t(x,\cdot)$.


A measure $\nu$ which is $p_t$ invariant is also $T_t$ invariant in the sense of our definition of invariance for semigroups acting on $\mathcal{B}(\mathcal{E})$.
Note that $\nu = \mu$ is invariant under $p_t$ since $ \int p_t(x,B) \mu(dx) = \mu(B)$, since the left hand side is equal to $$\int \mathbbm{1}_B(y) p_t(x,dy) \mu(dx) = \int \left(p_t \MyChi_B\right)(x) \mu(dx) = \langle \mathbbm{1}, p_t \MyChi_B \rangle_\mathcal{H} = \langle p_t^\star \mathbbm{1}, \MyChi_B \rangle_\mathcal{H} = \langle p_t \mathbbm{1} , \MyChi_B \rangle_\mathcal{H} = \mu(B)\;,$$ where we have used  both $p_t^\star = p_t$ and $p_t \mathbbm{1} = \mathbbm{1}$.

To a self-adjoint positive operator $-A$ in a real (or complex) Hilbert space $\mathcal{H}$ there is uniquely associated a closed bilinear (resp. sesquilinear) positive form
 $\mathbcal{E}_\mathcal{H}$  on $\mathcal{H} \times \mathcal{H} $ such that $\langle (-A)^{\frac{1}{2}} f , (-A)^{\frac{1}{2}} g \rangle =
 \mathbcal{E}_\mathcal{H}(f,g)$, for all $f,g \in D\left(\mathbcal{E}_\mathcal{H}\right)= D\left( (-A)^{\frac{1}{2}} \right)$,  $D\left(\mathbcal{E}_\mathcal{H}\right)$ being the (dense) domain of the form as a dense subset of $\mathcal{H}$, e.g., \cite{Kato}

Especially $D(-A) \subseteq D\left( (-A)^{\frac{1}{2}}\right)$,  $(-A)^{\frac{1}{2}}$ is defined, e.g., by the spectral theorem. If $-A$ is only symmetric, positive,
 then $(f,-A g) =  \dot{\mathbcal{E}}_\mathcal{H} (f,g)$ for any $f$ in some minimal domain $D\left( \dot{\mathbcal{E}}_\mathcal{H} \right) $, $g \in D(A)$.
If a sesquilinear form has this aspect then it is automatically
closable on $D\left( \dot{\mathbcal{E}}_\mathcal{H}
\right)\subset\,D(A)$, see \cite[Th. 1.2.7]{Kato}. There is a very
interesting relationship between self-adjoint $C_0-$contraction
semigroups, their positive generators and special symmetric closed,
positive sesquilinear forms. For this we take $\mathcal{H} =
L^2_\mathbb{R} (\mathcal{E}, \mu)$, for some $\sigma-$finite space
$\left(\mathcal{E}, \mathcal{B}(\mathcal{E}), \mu\right)$. A closed
symmetric positive sesquilinear form acting on $\mathcal{H} \times
\mathcal{H}$ is said to be  a Dirichlet form if it has the
contraction property $\mathbcal{E}_\mathcal{H}\left( f^\#,
g^\#\right) \leq \mathbcal{E}_\mathcal{H} (f,g)$ for $f^\# :=
\left(f \vee 0\right) \wedge 1$, $f,g \in
D(\mathbcal{E}_\mathcal{H})$. It turns out that such forms are in
$1-1$ correspondence with self-adjoint Markov semigroups $T_t$ on
$\mathcal{H}$.

The relation is characterized by $\mathbcal{E}_\mathcal{H} (f,g) = \left( (-A)^{\frac{1}{2}} f , (-A)^{\frac{1}{2}} g\right),$ with $-A$ the infinitesimal generator of $T_t$. The theory of Dirichlet forms describes these relations and gives a precise description of Markov processes associated with such structures. The properties of the associated Markov processes depend on $regularity$, resp. $quasi-regularity$, of the underlying Dirichlet forms, see, e.g., \cite{FuOTa,MaRo}


\section{Invariant measures in finite dimensions}\label{Section2}
\subsection{The case of  Ornstein-Uhlenbeck L\'{e}vy processes}\label{Section2.1}
The aim of this section is
to characterize the invariant measure corresponding to the solution
of the following finite dimensional SDE
\begin{align*}
    \dd X(t) = A X(t) \dd t + \beta(X(t))\dd t + \dd L(t),
\end{align*}
where $A$ is a positive definite matrix on $R^d$, $\beta: \R^d \to
\R^d$ is a possibly nonlinear function from $\R^d$ into itself and
$L(t)$ is an $\R^d$-valued L\'evy process generated by the triplet
$(Q,\nu, \gamma)$ (see below and  \cite[Definition 8.2]{Sat91} for
more details). To this end, we will first recall some well-known
result concerning the description of the invariant measure
corresponding to the Ornstein-Uhlenbeck process on $\R^d$. We refer
to \cite[Chapter 17]{Sat91} and \cite[Sections 2,3]{SY84} for a more
complete treatment of the subject.


We recall that a
probability measure $\mu$ on $\R^d$ is infinitely divisible if and only if
its Fourier transform $\hat{\mu}$ has the L\'evy-Khinchine form
\begin{equation}\label{RAHH}
\hat{\mu}(z)={\rm exp}\left\{-{1 \over 2}\langle z,\,Qz\rangle +
i\langle \gamma\,,z\rangle + \int_{\R^d}\left(
\mathrm{e}^{\mathrm{i}\langle z,x \rangle} - 1 - \mathrm{i}\langle
z,x \rangle\,\MyChi_{B_1}(x) \right)\,\nu(dx)\right\}  \ z\,\in\,\R^d,
\end{equation}
where $Q$ is a symmetric positive definite $d\times d$-matrix,
$\gamma\,\in\,\R^d$, $\nu$ is a (non-necessarily finite, but
positive) $\sigma-$finite measure on $\R^d$ satisfying
$\nu(\{0\})=0$, and $\int \left( x^2 \wedge 1 \right) \nu(dx) <
+\infty$, where $B_1$ is the unit ball in $\R^d$, see,
e.g.,\cite[Theorem 8.1]{Sat91}. Such a measure $\nu$ is called
L\'{e}vy measure of $\mu$.\\ Following \cite{Sat91}, we call
$(Q,\,\nu,\,\gamma)$ the generating triplet (or simply the
characteristics) of $\mu$, as in \cite{APD}. $Q$, $\nu$, $\gamma$
are called respectively the Gaussian covariance matrix, the L\'evy
measure and the drift of $\mu$. We notice that when $Q=0$, $\mu$ is
called purely non Gaussian. When $Q=0, \gamma=0$ then $\mu$ is said
to be of purely jump-type.
The term $\psi_{\nu}(z):={\int_{\R^d}\left(
\mathrm{e}^{\mathrm{i}\langle z,x \rangle} - 1 - \mathrm{i}\langle
z,x \rangle\,\MyChi_{B_1}(x) \right)\,\nu(dx)}$ is often called
``characteristic exponent" or ``L\'evy symbol" or ``L\'evy
exponent".
\begin{rem}\label{BAABA}
The form of the jump-type term in the formula \eqref{RAHH} for the
Fourier transform of $\mu$ can also, equivalently, be written as
\begin{equation}
{\mathrm exp}\left\{\int_{\R^d}\left( \mathrm{e}^{\mathrm{i}\langle z,x
\rangle} - 1 - \mathrm{i}\langle z,x \rangle\,c(x) \right)\,\nu(dx)\right\} \; , \; z \in \mathbb{R}^d \, ,
\end{equation}
for any bounded measurable real-valued function $c(x)$ on $\R^d$,  such that $ x \mapsto \mathrm{e}^{i \langle z,x \rangle} - 1 -
\mathrm{i}\langle z,x \rangle\,c(x)$ is in $L^1 (\R^d,\,\nu)$ and
$c(x)=O({1 \over{\mid x\mid}})$ as $|x|\longrightarrow\,
\infty$,
provided we replace simultaneously $\gamma$ by $\gamma_c=\gamma+
\int_{\R^d}\left( c(x)-\MyChi_{B_1 }(x) \right)\nu(dx)$,
\\
A frequently used choice of $c(x)$ is $c(x)= {1 \over {1 +\mid
x\mid^2}}$, with $x\in \R^d$. For this and other choices for $c$, see, e.g.,
\cite[pgg. 38,39]{Sat91}. One characterizes the L\'evy-Khinchine
formula rewritten in this term as L\'evy-Khinchine formula with
generating triplet $(Q,\,\nu,\,\gamma_c)$.
\end{rem}


L\'{e}vy processes constitute the natural class of stochastic processes $L(t)$ associated with infinitely divisible probability measures on $\mathbb{R}^d$. We simply recall that they are characterized by having independent stationary increments and they satisfy $L(0) =0$ a.s., are stochastically continuous (i.e. continuous in probability, namely $\mathbb{P} (\mid L(t) -L(s) \mid > \epsilon ) \rightarrow 0$ as $t \downarrow s$, for all $\epsilon >0$) and c\`{a}dl\`{a}g (right continuous paths, with left limits, a.s.). Their transition functions are of the form $p_t^L (x,B) = p_1(B-x)^t$, the t-th convolution power of $p_1(B-x)$, where $p_1(B) := p_{L(1)}(B)$, i.e. $p_1(\cdot)$ is the law of $L(1)$.

We say that $L(t)$ corresponds to the infinitely divisible distribution $p_{L_1}$ on $\left(  \mathbb{R}^d , \mathbb{B}(\mathbb{R}^d)\right)$ or it is generated by the triplet
 $\left( Q, \nu, \gamma \right) $ of $p_{L_1}$.
Define the corresponding Markov semigroup $p_t^L$ by $\left(p_t^L f\right) (x) = \int_{\mathbb{R}^d} f(y) p_t^L (x,dy)$, for $f \in \mathbb{B} (\mathbb{R}^d)$.
We can restrict it to the Banach subspace $C_0 (\mathbb{R}^d)$ of functions vanish at infinity, with supnorm, since indeed it leaves $C_0(\mathbb{R}^d)$ invariant, see \cite[pp.207-208]{Sato}.

One has that
\begin{equation}
 \left( p_t f \right) (x) = \mathbb{E} f ( x+L(t)) = \int_{\mathbb{R}^d} \left( p_{L_1} (dy)\right)^t f(x+y) \;,\; f \in \mathbb{B}(\mathbb{R}^d)\;,\; x \in
 \mathbb{R}^d.
\end{equation}

 For $f$ of the form $f_z(x) = e^{i \langle z,x
\rangle}$ with $x,z \in \mathbb{R}^d$, we have then
\begin{equation}
\mathbb{E} \left( f(x+L(t) ) \right) = \mathbb{E} \left( e^{i \langle z, x+L(t) \rangle}\right)= \int_{\mathbb{R}^d} p_{L_1} (d\rho)^t e^{i\langle z,x+\rho\rangle}\;,
\end{equation}
hence for $x=0$, the definition of Fourier transform and (2), the
following holds
\begin{equation}
\begin{split}
\mathbb{E} \left( e^{i \langle z, L(t) \rangle}\right) &= \left(\widehat{p_{L_1}} (z) \right)^t\\
&={\exp}\left\{-{t \over 2}\langle
z,\,Qz\rangle + it\,\langle
\gamma,\,z\rangle + t  \int_{\R^d}\left(
e^{i \langle z,y \rangle} - 1 - i\langle
z,y \rangle\,\chi_{B_1}(y) \right)\,\nu(\dd y)\right\}\;.
\end{split}
\end{equation}

In particular one thus gets, for any $x \in \R^d$:
\begin{equation}\label{bbbb}
\mathbb{E}(e^{i\langle x,\,L(t)\rangle})={\rm e}\left\{-{t \over
2}\langle x,\,Qx\rangle + it\,\langle \gamma,\,x\rangle + t \cdot
\int_{\R^d}\left( e^{i \langle x,y \rangle} - 1 - i\langle x,y
\rangle\,\chi_{B_1}(x) \right)\,\nu(\dd y)\right\}.
\end{equation}
The infinitesimal generator $\mathcal{L}$ of $P_t,\,t\,\geq 0$ (and
of $(L(t))_{t\geq 0})$ has $C^{\infty}_{0}(\R^d)$ as a core (i.e.,
it is the closure in $C_{0}(\R^d)$ of its restriction to
$C^{\infty}_{0}(\R^d)$) and on $C^2_{0}(\R^d)$ it acts as
\begin{multline}\label{eq:inf_genLevy}
L f(x)= {1 \over 2}\,\ds_{j,k=1}^d q_{j,k}\,{\partial
\over{\partial_{x_j}\partial_{x_k}}}\,f(x)
 +\langle\gamma\,,\nabla f(x)\rangle\,+
\\ +\int_{\R^d}\left( f(x+y)-f(x)-\chi_{B_1}(y)\langle y,\,\nabla
f(x)\rangle \right)\,\nu(\dd y), \qquad f\in \,C^2_{0}(\R^d),
\end{multline}
where $(q_{j,k})_{j,k=1,\cdots,d}$ denotes the elements of the matrix $Q$.
More details can be found in  \cite[Theorem 31.5, p. 208]{Sat91}.\\
We shall now discuss perturbations of this semigroup and the
corresponding process by drift terms, beginning with the simple case
of a linear drift of a special form, passing then to a general
linear
drift and finally to the case of a nonlinear drift.\\
In the next proposition we shall show that starting from a L\'evy
process $(L(t))_{ t\geq 0 }$ one can construct the
transition probability function for an Ornstein-Uhlenbeck process
with parameter $c >0$ and L\'evy noise $L(t)$. In particular we will see that, defining
$X^c(t):=e^{-c\, t}+ \int_0^t\,e^{-c(t-s)} \dd L(s)$ for any $t \geq 0$, then $ X^c(t)$ is the unique mild solution
of the linear SDE with L\'evy noise
\begin{equation}
\begin{aligned}\label{ABDRA}
\dd X^c(t)&=-c\,X(t)\dd t+\dd L(t),\qquad t\geq 0.\\
X^c(0)&=x.
\end{aligned}
\end{equation}
For any $c>0$ we will denote by $\mathcal{L}^{c}$ the infinitesimal
generator  of the temporally homogeneous transition semigroup
$p_t^{c}$ of $X^c(t),$ defined first on
$C^2_{0}(\R^d)\,\subset\,C_{0}(\R^d)$; it turns out that
$\mathcal{L}^{c}$ has on $C_{0}^2(\R^d)$ the form $\mathcal{L} + c
\cdot\,\nabla$, where $\mathcal{L}$ is the linear operator defined
on $C_{0}^2(\R^d)$ in \eqref{eq:inf_genLevy}
\begin{prop}
    \label{thm:satz1}
    Let $(L(t))_{t\geq0}$ be a $d$-dimensional time homogeneous, L\'{e}vy process on $\mathbb{R}^d$, generated by a triplet $(Q,\nu,\gamma)$.
    Let $c\,> 0$. Then there is a temporally homogeneous transition probability function $( p_t^{c})_{t\geq 0}$ on $\mathbb{R}^d \times \mathcal{B}(\R^d)$ such that
    \begin{equation}
        \int_{\mathbb{R}^d} \mathrm{e}^{\mathrm{i}\langle z,y \rangle}p_t^{c}(x,\mathrm{d}y) = \exp \left[ \mathrm{ie}^{-c\,t}\langle x,z \rangle +
        \int\limits_{0}^{t}\psi(\mathrm{e}^{-c\, s}z)\mathrm{d}s \right], \qquad x, z\in
        \R^d,
    \end{equation}
    with $\psi(z):=\log\hat{p}^{c}_{L_1}(z)$, $ z\in\R^d$. $p_t^{c}(x,\mathrm{d}y)$ is the transition function of the OU process with L\'evy noise $L(t)$ associated to
    the equation \eqref{ABDRA}.

    For each $t\geq 0,x\in\R^d$, the probability measure $B \mapsto p_t^{c}(x,B)$ is an infinitely divisible probability measure on $\mathbb{R}^d$ with generating triplet
    $(Q_t,\nu_t,\gamma_{t,x})$ given by
\begin{equation}
\left\{
\begin{array}{lll}
        Q_t := \int\limits_{0}^{t}\mathrm{e}^{-2cs}\,\mathrm{d}s\,Q, \\
        \nu_t(B) := \int\limits_{\R^d} \nu(\mathrm{d}y) \int\limits_{0}^{t} \chi_B(\mathrm{e}^{-c\, s}y)\,\mathrm{d}s, \quad B\in\mathcal{B}
        ({\R}^d), \\
        \gamma_{t,x} := \mathrm{e}^{-c\, t}x + \int\limits_{0}^{t}\mathrm{e}^{-c\, s}\,\mathrm{d}s\,\gamma + \int_{\R^d}\int\limits_{0}^{t}\Big(
        \mathrm{e}^{-cs}y[\chi_{B_1}(\mathrm{e}^{-cs}y) - \chi_{B_1}(y)]\,\mathrm{d}s\,\Big)\,\nu(\mathrm{d}y),
\end{array}
\right.
\end{equation}
\end{prop}
\begin{proof}
    The proof is in \cite[Lemmas 17.1 and 17.4]{Sat91}.
\end{proof}

\begin{rem}
    Proposition \ref{thm:satz1} extends to separable Hilbert spaces $\mathcal{H}$ using basic properties of measures on $\mathcal{H}$, see, e.g., \cite{Part}.
\end{rem}

\begin{rem}
    When $L(t)$ is the standard Brownian motion on $\mathrm{R}^d$ the temporally homogeneous Markov process having the transition function
     $(p_t^{c})_{t\geq 0}$ of Proposition (\ref{thm:satz1}) is just the Ornstein-Uhlenbeck process on ${\R}^d$ (with ``diagonal" drift $b(x)= -c\,x,\,x\,\in\,\R^d, c>0$).
\end{rem}

By definition, in the general case of the Proposition
(\ref{thm:satz1}), where $L(t)$ is a general L\'{e}vy process on
$\mathbb{R}^d$ with L\'evy triplet $(Q,\nu,\gamma)$, the temporally
homogeneous Markov process $Y(t)$ with transition function
$(p_t^{c})_{t\geq 0}$ is called the Ornstein-Uhlenbeck
process with
L\'{e}vy noise $L(t)$ (or process of Ornstein-Uhlenbeck-type generated by $(Q,\nu,\gamma,c)$, in the terminology of \cite[Definition 17.2]{Sat91}).

Similarly as for the above derivation of the formula (\ref{bbbb}) for
$\mathbb{E}(e^{i\langle x,\,L(t)\rangle})$ starting from $p_t$ we
derive the following:
\begin{equation*}
\begin{aligned}
(P_t^{c}
f)(x):&=\mathbb{E}^x(f(X(t)))& \\
&=\int_{\R^d} p_t^{L,c}(x,\,\dd y)\,f(y)\nonumber\\&=\,\int_{\R^d}
p_t^{L,c}(x,\,\dd y)\,f(e^{-c\,t}\,x+y),\qquad for \ any \,f \in
C_0(\R^d), x,y \,\in\,\R^d.
\end{aligned}
\end{equation*}
In the above formula $\mathbb{E}^x$ stands for the expectation
with respect to the underlying measure for the process $X(t),t\geq 0, $ started at $x$.\\
We get
\begin{equation}
\mathbb{E}^x(e^{i\,\langle y,\,X(t)\rangle})={\rm exp}\left\{ i\,e^{-c\, t}\,\langle
y,\,x\rangle + \int_0^t\,\psi(e^{-c(t-s)}\,x)\dd s\right\},\,x,\,y\,\in\,\R^d.
\end{equation}
\Big(with, as in Proposition 2.2,\,
$\psi(z):=\log\hat{p}^{c}_{L_1}(z),\, z\in\R^d$\,\Big).

These considerations have been extended in \cite{ SY84} to the case
of general linear drift terms of the form $-A\cdot\,\nabla$, with
$A$ a
non-negative symmetric real-valued $d\times d-$matrix.\\
The analogue of Proposition (\ref{thm:satz1}) holds with $c$
replaced by $A$, $e^{-ct}\langle x,z\rangle$ by $ \langle
e^{-At}x,z\rangle$, $\psi(e^{-c\, s}z)$ by $\psi(e^{-As}z)$. Moreover,
corresponding formulas for $(Q_t, \nu_t, \gamma_{t,x})$ hold with
$e^{-{2c\, s}}$ and $e^{-c\, s}$ replaced respectively by $e^{-{2As}}$,
$e^{-As}$. For the proof we refer to \cite{SY84}.\\
Also the formulae for $P_t^{c}$ and $\mathcal{L}^{c}$ extend
correspondingly to formulae for the corresponding quantities
$P_t^{A}$ and $\mathcal{L}^{A}$, as follows:
\begin{prop}
    \label{thm:satz2}
    The smallest closed extension of $\mathcal{L}^{A}$ in $C_{0}(\mathbb{R}^d)$ is the infinitesimal generator of a strongly continuous non-negative semigroup $(P_t^{A})_{t\geq0}$, such that
    \begin{equation}
        (P_t^{A}f)(x) = \int\limits_{\mathbb{R}^d} f(y)p_t^A(x,\mathrm{d}y),
    \end{equation}
    where $(p_t^{A}(x,\cdot))_{t\geq 0,x\in \R^d}$ are the transition probabilities of the $\R^d$-valued process solving \begin{equation}\label{mmmm}
  \mathrm{d}X(t)= -AX(t)\,\mathrm{d}t+\mathrm{d}L(t),\\with\,\,
  X(0)=x,\,x\,\in\,\R^d,\,t>0.
    \end{equation}

\end{prop}

One has that $P_t^{A}$ maps $C_{0}(\mathbb{R}^d)$
\,into it self and $$ \| P_t^A \|:=\sup_{\|f\|_u\leq 1} |f(x)| =1,$$
for any $t\geq0$. Moreover, for each $t>0$ and $x\in\mathbb{R}^d$,
$p_t^A(x,\,\cdot\,)$ is an infinitely divisible distribution such
that
\begin{equation}
    \hat{p}_t^{A}(x,z) = \mathrm{exp}\left\{i \langle x,\mathrm{e}^{-tA}z \rangle+ \int_{0}^{t}\log\hat{p}_{L_1}(\mathrm{e}^{-sA}z)
    \,\mathrm{d}s\right\},\,x,\,z\,\in\,\R^d.
\end{equation}
In particular, the generating triplet of $p_t^{A}(x,\,\cdot\,)$ is
an infinitely divisible distribution and is given by
$(Q_{t},\nu_t,\gamma_{t,x})$, where
\begin{equation}
\label{7.1eqa}\left\{
\begin{array}{lll}
    Q_t := \int\limits_{0}^{t} \mathrm{e}^{-sA}Q\mathrm{e}^{-sA}\,\mathrm{d}s, \\
    \nu_t(B) := \int_B\,(\int\limits_{0}^{t} \chi_{B_1}(\mathrm{e}^{-sA}x)\,ds)\,\nu(dx)\\
    \gamma_{t,x} := \mathrm{e}^{-tA}x + \int\limits_{0}^{t} \mathrm{e}^{-sA}\gamma\,\mathrm{d}s + \int\limits_{\mathbb{R}^d} \int\limits_{0}^{t} \mathrm{e}^{-sA}z
    \{ \chi_{B_1}(\mathrm{e}^{-sA}z) - \chi_{B_1}(z) \} \,\mathrm{d}s\,\nu(\mathrm{d}z).
\end{array}
\right.
\end{equation}

This process $X(t)$ is proven to have a modification $\tilde{X}(t)$
with c\`{a}dl\`{a}g paths (i.e.\ $P(X(t)=\tilde{X}(t))=1$ for all
$t\in[0,\infty)$, and $\tilde{X}_t$ is c\`{a}dl\`{a}g), see, e.g.,
\cite[Theorem 3.7]{Dyn65I}, \cite{Dyn65II,EthKur86,Chu82}.
Of course the classical Ornstein-Uhlenbeck process has a
modification with continuous paths.

For the generator $\mathcal{L}^{A}$    of the corresponding transition semigroup $P_t^{A}$ we have:
\begin{equation}
\mathcal{L}^{A} f(x)=\mathcal{L} + A \cdot \nabla,
\qquad \,\,on\,\,C^2_{0}(\R^d)\,\subset\,C_{0}(\R^d),
\end{equation}
where $\mathcal{L}$ has been defined in \eqref{eq:inf_genLevy}.
Moreover the formula for the characteristic function of $X(t)$ (solution of \ref{mmmm}))
becomes:
\begin{equation}
\mathbb{E}^x(e^{i\,\langle z,\,X(t)\rangle})={\rm exp}\left\{ i\,e^{-At}\,\langle
z,\,x\rangle + \int_0^t\,\psi(e^{-A(t-s)}\,z)\dd s\right\},
\end{equation}
with $\psi(z):=\log\hat{p}_{L_1}(z)$, for any $z\in\R^d$, as in
Proposition 2.2.

We shall now discuss the situation where there is an invariant
measure for the OU processes considered above, i.e. both $X^c$ and $X$. We start by $X^c(t)$.



\begin{prop}[{\cite[Theorem 1.75]{Sat91}}]
    \label{thm:satz3}
    Let $L(t)$ be as in Proposition \ref{thm:satz1} . If its L\'evy measure $\nu$ satisfies
    \begin{equation}
        \int\limits_{\left| x \right|>2} \log\left| x \right|\nu(\mathrm{d}x) < \infty
    \end{equation}
    then the Ornstein-Uhlenbeck process $X^c(t)$ on $\mathbb{R}^d$ with L\'{e}vy noise given by $L(t)$, generated by $(Q,\nu,\gamma,c)$, $c>0$ and solving
    (\ref{eq:inf_genLevy}), has a limit distribution for $t\,\longrightarrow\, +\infty $
given by
    \begin{equation}
        \hat{\mu}(z) =
\mathrm{exp}\left\{\int_{0}^{\infty}\psi(\mathrm{e}^{-c\, s}z)\,\mathrm{d}s\right\}, \quad z\in\mathbb{R}^d.
    \end{equation}
    This measure $\mu$ is self-decomposable (and in particular infinitely divisible), i.e.\ it satisfies the property that $\hat{\mu}(z)=\hat{\mu}(b^{-1}z)\hat{\nu}_b(z)$, for any $b>1$ and some probability measure $\nu_b$ on $\mathbb{R}^d$. \\
    The generating triplet $(Q_\infty,\nu_\infty,\gamma_\infty)$ of $\mu$ is given by
   \begin{equation}
\left\{
\begin{array}{lll}
        Q_\infty := \frac{1}{2c}Q \\
        \nu_\infty(B) := \frac{1}{c} \int\limits_{\mathbb{R}^d} \nu(\mathrm{d}y) \int\limits_{0}^{\infty} \chi_B(\mathrm{e}^{-s}y)\,\mathrm{d}s,
         \quad B\in\mathcal{B}(\mathbb{R}^d), \\
        \gamma_\infty\,:= \frac{\gamma}{c} + \frac{1}{c} \int\limits_{\left| y \right|>1} \frac{y}{\left| y \right|}\nu(\mathrm{d}y).
\end{array}
\right.
\end{equation}
\end{prop}

\begin{proof}
    See \cite[Theorem 17.5 $i)$]{Sat91}.
\end{proof}

\begin{rem}
    In \cite[Theorem 17.5]{Sat91} a converse of this proposition is also proven.
\end{rem}

\begin{theo}\label{RAHBOU}
    An Ornstein-Uhlenbeck process with L\'{e}vy noise $L(t)$ satisfying the assumptions of Proposition \ref{thm:satz3} has a unique invariant invariant measure and this invariant measure is self-decomposable.
\end{theo}

\begin{proof}[Proof ({\cite[page 112]{Sat91}})]
    From Proposition \ref{thm:satz3} there is a limit self-decomposable distribution $\mu$. \\
    On the other hand from the semigroup property of $(p_t)_{t\geq 0}$ (Chapman-Kolmo\-gorov equation) we have
    $\int_{\R^d} p_s(x,\mathrm{d}y)\int_{\R^d} p_t(y,\mathrm{d}z)f(z)=\int_{\R^d} p_{s+t}(x,\mathrm{d}z)f(z),\,f\,\in\,C_b(\R^d)$ and the continuity of
    $x\to\int p_t(x,\mathrm{d}z)f(z)$ as an operator on $C_b(\mathbb{R}^d)$, we have
    \begin{align*}
  \lim_{s\to \infty} \int_{\R^d} p_s(x,\mathrm{d}y)\int_{\R^d} p_t(y,\mathrm{d}z)f(z) = \int_{\R^d} \mu(\mathrm{d}y)\int_{\R^d} p_t(y,\mathrm{d}z)f(z)=\int_{\R^d}\,\mu(\mathrm{d}z)f(z),
\end{align*}
which shows that $\mu$ is invariant.

Uniqueness is shown by proving that if $\tilde{\mu}$ is another invariant measure then
$$\lim_{t \to \infty} p_t^* \tilde{\mu} = \tilde{\mu},$$
with $p_t^* $ the adjoint of $p_t $, and taking $t \rightarrow
+\infty$ we get $\int_{\R^d} f(y)\mu(\dd y) = \int_{\R^d} f(y)\tilde{\mu}(\dd y),\,for\,any\,\,f\,\in\,C_b(\mathbb{R}^d), $ i.e. $\mu =
\tilde{\mu}$.
\end{proof}
\begin{rem}
  As shown by \cite[Theorem 17.11]{Sat91} the condition in Theorem \ref{RAHBOU} is also necessary for having an invariant distribution.
\end{rem}

Now we turn to the existence and uniqueness of an invariant measure for the OU L\'evy process with drift coefficient $-A$, with $-A$  a non-negative symmetric real valued $d\times d$-matrix, i.e. to the process $X$ corresponding with equation \eqref{mmmm}.
We quote from \cite{SY84} the following result.
   \begin{prop} \label{PropositionA}
     Let $A$ be a real $d \times d$ matrix whose eigenvalues possess positive real parts.
If the L\'evy measure of the $L(t)$ of Proposition (\ref{thm:satz1}) satisfies
\begin{equation}\label{SERB}
     \int_{|y|> 1} \log|y| \nu (\dd y) < \infty,
\end{equation}
   then there exists a limit distribution $\mu$ for $(p_t^{A})_{t\geq 0}$ (with $p_t^{A}$ as in Proposition \ref{thm:satz2}).
   Moreover, $\mu$ is $Q$-selfdecomposable 
and is the unique invariant measure for the solution $X$ of equation \eqref{mmmm}, i.e.
    the Ornstein-Uhlenbeck process with drift coefficient $-A$ and L\'evy noise $L(t)$.

    In particular we have
\begin{equation}
     \hat{\mu} (z) = e^{ \int_0^\infty \log \hat{p}_{L_1} \left( e^{-s A^\star} z\right) ds} \:,
\end{equation}
    with $A^\star$ being the adjoint of $A$.

    The generating triplet for $\mu$ is thus given by $\left( Q_\infty, \nu_\infty,\gamma_\infty\right)$, where
\begin{equation*}
\begin{aligned}
 Q_\infty &= \int_0^\infty e^{-s A} Q e^{-s A^\star} \dd s,\\
    \nu_\infty\left(B\right)& = \int_B \int_0^\infty \left( \chi_{B_1}(e^{s A} x) \right) \dd s\,\nu(\dd x),\, B \in \mathcal{B}(\R^d),\\
  \gamma_\infty &= A^{-1}\gamma + \int_{\R^d} \int_0^\infty e^{-s A} z
\left(\chi_{B_1(0)} (e^{-s A} z) -\chi_{B_1(0)}(z)\right) \dd
s\, \nu (\dd z).
\end{aligned}
 \end{equation*}
    Conversely, every $Q$-selfdecomposable distribution can be realized in this way. The correspondence between $\mathcal{L}^A$ and $\mu$ is $1$-$1$.
    \end{prop}
    \begin{proof}
    See \cite[pgg. 77--99]{SY84}.
    \end{proof}

\begin{rem}
   \begin{itemize}
     \item[(1)]  If $\mu$ is infinitely divisible and is not a delta-distribution, then its support is unbounded (see \cite[ Corollary. 24.4]{Sat91}).
    \item[(2)] The condition \eqref{SERB} in Proposition \ref{PropositionA} is
    necessary. If it is not satisfied then the process has no
    invariant measure, see. \cite[Theorem 4.2]{SY84}.
    \item[(3)] If $\mu(a+V) < 1$ for any $a \in \R^d$ and any subspace $V \subset \R^d$ with $dim(V) \leq d-1$, (i.e. $\mu$ is non degenerate),
    then $\mu$ is absolutely continuous with respect to the Lebesgue measure on $\R^d$ (see \cite{Yama}). Nondegeneracy of $\mu$ is equivalent with
    $|\widehat{\mu} (z) | \leq 1- c_1 |z| ^2$, for any $|z| < c_2$, for some $c_1,c_2 >0$ (see \cite[Proposition 24.19]{Sat91}).
  \end{itemize}
\end{rem}
\begin{rem}
 See \cite[pag. 117-118]{Sato} for history of these results and additional references. See also \cite{SY84}
 for a very interesting survey of selfdecomposability and selfsimilarity with applications to Orstein-Uhlenbeck processes with L\'evy noise.

For criteria for selfdecomposability of measures on $\R^d$ see,
e.g., in \cite[Theorem 15.10]{Sat91}: they only involve the L\'evy
measure $\nu$. An example of a process of Ornstein-Uhlenbeck with
L\'evy noise having strictly $\alpha$-stable distribution $\mu$ is
given in \cite[Theorem 4.2]{CuPe}.
For $c= \frac{1}{\alpha}, \, \alpha > 0$, defining $Y(t) = e^{-
\frac{t}{\alpha}} L(e^t)$ we have for any $t_0$, that $X(t_0+t)$, $
t \geq 0$ is an Ornstein- Uhlenbeck process of L\'evy type
(associated with $L(t)$ and c), and $p_{L(1)} = p_{X(t)}$, for all $
t \geq 0$ (see \cite{Bre68,Bre70}). The condition in Proposition
\ref{thm:satz3} implies that the associated Ornstein-Uhlenbeck
process with L\'evy type process $X(t)$ is recurrent (cfr. \cite[p.
272]{Sat91}).
\end{rem}


\subsection{Perturbations by non linear drifts: an analytic approach}\label{Section2.2}

Let $\mu$ be a probability measure on $\mathbb{R}^d$. At the
beginning of section $2.1$ we recalled that, if $(P_t)_{t\geq 0}$ is
a one parameter strongly continuous contraction semigroup on
$L^2(\mu)$, then the measure $\mu$ is invariant for $(P_t)_{t\geq
0}$ if
\begin{equation*}
    \dint_{\mathbb{R}^d}(P_tf)(x)\mu(\dd x) = \dint_{\mathbb{R}^d}f(x)\mu(\dd x), \qquad \forall \ f \in
    L^2(\mu).
\end{equation*}
 This in turn is equivalent to:
\begin{equation*}
    P_t^{\ast}1=1\,,\quad \forall t \geq 0,
\end{equation*}
where $P_t^{\ast}$ is the adjoint semi-group acting in $L^2(\R^d;
\dd \mu)$ and $1$ is the function identically 1 in $L^2(\mu)$.
If $L_0$ is an operator in $L^2(\R^d; \dd\mu)$ defined on a dense domain $D(L_0)$ then $\mu$ is said to be $\left( L_0 , D(L_0) \right)$-invariant if $\int_{\R^d} L_0f\,d\mu=0$, for all $f\in D(L_0)$. If $L$ with domain $D(L)$ is the generator of a one parameter strongly continuous contraction semigroup $(P_t)_{t\geq0}$ on $L^2(\R^d, \dd\mu)$ and if $\mu$ is $\left(L,D(L)\right)$-invariant then $\mu$ is also said to be  infinitesimal invariant under $(P_t)_{t\geq 0}$. \\
Note that invariance implies infinitesimal invariance, but in
general infinitesimal invariance does not imply invariance except
for symmetric processes, see, e.g., \cite{ABRW,Bhaka,BeSc,Eche,AMR}.

Consider the L\'{e}vy type operator $\left(L_0,S(\R^d)\right)$ acting on
$\mathcal{S}(\mathbb{R}^d)$ functions:
\begin{equation}\label{eqn:levyoperator}
    (L_0f)(x) = a_1(\Delta f)(x) + \beta(x)(\nabla f)(x) + a_2 \int\limits_{\mathbb{R}^d}[f(x+y)-f(x)]\nu_{\alpha}(\dd y)
\end{equation}
where $a_1\geq0$, $a_2\geq0$, $a_1+a_2>0$,
$\beta:\mathbb{R}^d\to\mathbb{R}^d$ is Borel measurable, locally
Lipschitz bounded and such that the Fourier transform $\hat{\beta}$
of $\beta$ exists and $\nu_{\alpha}(\dd y):=\frac{\dd y}{\left| y
\right|^{d+\alpha}}$, $\alpha\in(0,2)$ is a stable L\'{e}vy measure.

We recall that a stable L{\'e}vy process is a stochastic process
whose characteristic exponents correspond to those of distributions
$Y$ (they are called stable distributions, introduced by P. L{\'e}vy
in \cite{Le24} and \cite{Le25}) such that for all $ n \in \N$ the
following holds:
\begin{equation}
 \sum_{k=1}^n Y_k \stackrel{d}{=} \tilde{a}_n Y +\tilde{b}_n \; ,
\end{equation}
where $Y_1, \ldots, Y_n$ are independent copies of $Y$, while $\tilde{a}_n
>0$, $\tilde{b}_n$ are real constants. See, e.g., \cite{Sat91} for the
discussion of stable L{\'e}vy measure.

If $f$ is a function on $\mathbb{R}^d$ we define the Fourier
transform $\hat{f}$ of $f$, by:
\begin{equation}
\hat{f}(k)=\dint_{\mathbb{R}^d}e^{ikx}{f}(x)\,\dd x,\,\,k\,\,\in\,\R^d.
\end{equation}
similarly for $f(x)\,\dd x$ replaced by a measure $\nu$ respectively
a distribution, whenever the transforms exists, in the corresponding sense.
\begin{prop}
    Let $L_0$ be a L\'evy operator of the form \eqref{eqn:levyoperator} and let $\mu$ be a probability measure on $\mathbb{R}^d$.
    Then $L_0$ can be seen as a densely defined operator on $L^2(\mathbb{R}^d,\mu)$, with $D(L_0)=\mathcal{S}(\mathbb{R}^d)$.

    If $\widehat{\beta\mu}$ exists, then $\mu$ is $(L_0,S(\mathbb{R}^d))$-invariant if $\mu$ satisfies:
    \begin{align*}
        \dint_{\mathbb{R}^d}\hat{f}(k)\hat{L}_0(k)\hat{\mu}(\dd k) = \frac{i}{(2\pi)^{\frac{d}{2}}}\int\limits_{\mathbb{R}^d}\hat{f}(k)k\widehat{\beta\mu}(\dd k),
         \quad \forall f \in S(\mathbb{R}^d)
    \end{align*}
    where
    \begin{align*}
        \hat{L}_0(k) :&= \frac{1}{(2\pi)^{\frac{d}{2}}} \left[ -a_1\left| k \right|^2 + a_2c_{\alpha}\left| k \right|^{\alpha} \right],
        \qquad \alpha\in(0,2)\\
\intertext{and}
c_{\alpha}&=\,c_{\alpha}(u)\,\dint_{\mathbb{R}^d \backslash
        \{0\}}\,\cos{(\langle\,u,\,y\rangle-1)}\,\nu_{\alpha}(dy),
    \end{align*}
for some unit vector $u \in\R^d.$
\end{prop}

\begin{proof}
    The proof is given in \cite{ABRW} and \cite{BeSc} assuming $\mu$ has a density, and the general case is analogously proven.
\end{proof}

\begin{example}
    Let us take $a_1=0$, $\beta(x)=-x,\,\, x\,\in\,\R^d,$ and $L_0=a_2C_{\alpha}(-\Delta)^{\frac{\alpha}{2}}-x\cdot\nabla$ on $S(\R^d)$. \\
    The  $\left( L_0,D(_0L) \right)$ invariant measure is then given by $\mu(\dd x)=\rho_2(x)\dd x$ with $\hat{\rho}_2(k)=e^{-\frac{1}{\alpha}a_2c_{\alpha}\left| k \right|^{\alpha}},\,k\,\in\,\R^d$.
\end{example}
We shall now present a more systematic study of perturbation of
L{\'e}vy generators by non linear drifts using {\it ground state
transformations}, a concept which we first explain in the Gaussian
case:

\begin{prop} \label{prop:proposition1}
    Let $L_0$ be given by \eqref{eqn:levyoperator} with $a_2=0$ and $\beta(x)=-x,\,\,x\,\in\,\R^d$, i.e.\ $L_0=\Delta-x\cdot\nabla$, with domain
     $D(L_0)=\mathcal{S}(\mathbb{R}^d)$. Then:
    \begin{enumerate}
        \item\label{itm:item1} The adjoint of $L_0$ in $L^2(\mathbb{R}^d)$ is $\Delta+x\cdot\nabla+d$.
        \item\label{itm:item2} $\mu(\dd x)=\rho(x)\dd x$ with $\rho(x)=\frac{e^{-\frac{x^2}{2}}}{(2\pi)^{\frac{d}{2}}}$ is $\left(L_0,S(\R^d)\right)$-invariant.
        \item\label{itm:item3} The adjoint of $(L_0,D(L_0))$ in $L^2(\mathbb{R}^d,\mu),$ with $\mu$ as in 2., is equal to $L_0$ on $D(L_0)$.
         Thus $(L_0,D(L_0))$ is symmetric as an operator acting in $L^2(\mathbb{R}^d,\mu)$.
        \item\label{itm:item4}  The closure $\overline{L}_0$ with domain $D(\overline{L}_0)$ of $(L_0,D(L_0))$ in $L^2(\R^d, \mu)$ is self-adjoint in \\ $L^2(\mathbb{R}^d,\mu)$.
        \item\label{itm:item5}  $\mu$ is invariant under the strongly continuous contraction semigroup
        $e^{t\overline{L}_0},\,\,t\,\geq\,0$, in $L^2(\R^d, \mu).$
    \end{enumerate}
\end{prop}

\begin{proof}
   {\it Point \ref{itm:item1}}.
       For any $f,g\in\mathcal{S}(\mathbb{R}^d)$ we have, integrating by parts:
                    \begin{eqnarray}\label{eqn:equation1bis}
                        \int L_0f(x)g(x)\,\dd x &=& \int \left[ (\Delta-x\cdot\nabla)f(x) \right] g(x)\,\dd x \nonumber \\
                        &=& \int f(x)\Delta g(x)\,\dd x + \int f(x)\nabla(xg(x))\,\dd x \nonumber \\
                        &=& \int f(x)\Delta g(x)\,\dd x + \int f(x)(\nabla x)g(x)\,\dd x + \int f(x)x\nabla g(x)\,\dd x \hspace*{1cm} \label{eqn:equation2} \\
                        &=& \int f(x)\Delta g(x)\,\dd x + d \cdot \int f(x)g(x)\,\dd x + \int f(x)x\nabla g(x)\,\dd x, \nonumber
                    \end{eqnarray}
                    where we also used $\nabla x=d$.  This finishes the proof of \eqref{itm:item1}.

        {\it Point \ref{itm:item2}}. If we take $g=\rho$ in \eqref{eqn:equation2} we get
                    \begin{eqnarray}
                        \int L_0f(x)\rho(x)\,\dd x &=& \int L_0f(x)\mu(\dd x) \nonumber \\
                        &=& \int f(x)\Delta\rho(x)\,\dd x + d\int f(x)\rho(x)\,\dd x + \int f(x)x\nabla\rho(x)\,\dd x. \hspace*{8mm}\label{eqn:equation3}
                    \end{eqnarray}
                    But $\nabla\rho(x)=(-x)\rho(x)$,
                    \begin{equation}
                        \Delta\rho(x) = (-d)\rho(x)-x\nabla\rho(x) = (-d)\,\rho(x)+x^2\rho(x).
                        \label{eqn:equation4}
                    \end{equation}
                    From \eqref{eqn:equation3},\eqref{eqn:equation4} it follows
                    \begin{equation}
                        \int L_0f(x)\mu(\dd x) = \int f(x) \left[ (-d)\,\rho(x)+x^2\rho(x)+(d)\,\rho(x)-x^2\rho(x) \right]\,\dd x = 0.
                    \end{equation}
                    Hence $\mu$ is $(L_0,\mathcal{S}(\mathbb{R}^d))$-invariant.

        {\it Point \ref{itm:item3}}. We have, for any $f,g\in\mathcal{S}(\mathbb{R}^d)$, using \eqref{eqn:equation1bis} with $g$ replaced by $g\rho$:
                    \begin{eqnarray} \
                        \int (L_0f)(x)g(x)\rho(x)\,\dd x &=& \int f(x)(\Delta+x\cdot\nabla+d) \left( g(x)\rho(x)\right)\,\dd x \nonumber \\
                        &=& \int f(x)(\Delta g(x))\rho(x)\,\dd x + \int f(x)2\nabla g(x)\nabla\rho(x)\,\dd x\,\hspace*{1cm} \label{eqn:equation5} \\
                        &\quad +& \int f(x)g(x)\Delta\rho(x)\,\dd x + \int f(x)x(\nabla g(x))\rho(x)\,\dd x \nonumber \\
                        &\quad +& \int f(x)xg(x)\nabla\rho(x)\,\dd x + (d)\,\int g(x)\rho(x)\,\dd x \nonumber
                    \end{eqnarray}
                    Inserting the expressions \eqref{eqn:equation3} and \eqref{eqn:equation4} for $\nabla \rho$, resp. $\Delta\rho$,  into \eqref{eqn:equation5} we get:
                    \begin{eqnarray}
                        \int L_0f(x)g(x)\rho(x)\,\dd x &= \int f(x)\Delta g(x)\rho(x)\,\dd x + 2\int f(x)\nabla g(x)(-x)\rho(x)\,\dd x \nonumber\\
                        &\quad - (d)\,\int f(x)g(x)\rho(x)\,\dd x + \int f(x)x^2g(x)\rho(x)\,\dd x \nonumber\\
                        &\quad + \int f(x)x(\nabla g(x))\rho(x)\,\dd x + \int f(x)xg(x)(-x)\rho(x)\,\dd x \nonumber \\
                        &\quad + (d)\int g(x)\rho(x)\,\dd x \nonumber\\
                        &= \int f(x)\Delta g(x)\rho(x)\,\dd x - \int f(x)x\cdot\nabla g(x)\rho(x)\,\dd x,
                    \end{eqnarray}
                    which proves \ref{itm:item3}.

        {\it Point \ref{itm:item4}}. This is proven by the unitary ``ground state transformation''\ $U:L^2(\mathbb{R}^d)\longrightarrow\,L^2(\mathbb{R}^d,\mu)$
        defined by $f\in L^2(\mathbb{R}^d)\to U\,f\,\in L^2(\mathbb{R}^d,\mu)$, $Uf=\frac{f}{\sqrt{\rho}}$. \\
                    By this transformation we have, for any $f\in\mathcal{S}(\mathbb{R}^d)$:
                    \begin{equation}
                        U^{-1}(\Delta-x\cdot\nabla){U}f = (\Delta-x^2-d)f,
                    \end{equation}
                    as easily seen, and since $\Delta-x^2-d$ is essentially self-adjoint on $\mathcal{S}(\mathbb{R}^d)$
                    (the Hermite functions being analytic vectors for it), hence also the unitary equivalent operator
                    $\Delta-x\cdot\nabla$, restricted to $\mathcal{S}(\mathbb{R}^d)$ is essentially self-adjoint
                    (where we use that $U$ maps $\mathcal{S}(\mathbb{R}^d)$ into itself), hence its closure $\overline{L}_0$ is self-adjoint
                    (for such concepts see, e.g.,\cite{RS}).

         {\it Point \ref{itm:item5}}.  By the definition of $\mu$ invariant under $(P_t)_{t\geq 0}$ one has to prove $\int f\,d\mu=\int e^{t\overline{L}_0}f\,d\mu$, for all
        $t\geq0,f\in\mathcal{S}(\mathbb{R}^d)$. This can be proven by realizing that the
        right hand side is equal to $(e^{t\overline{L}_0}1,f)_{L^2(\mu)}$, where we used that $e^{t\overline{L}_0}$ is self adjoint, and $e^{t\overline{L}_0}1=1$, as seen by expansion in powers of $t$
        and using the fact that $\overline{L}_0^n1=0$, for all $n\in\mathbb{N}$.
\end{proof}
$\overline{L}_0$ is the well known generator of an
Ornstein-Uhlenbeck semigroup (and diffusion process) in $L^2(\R^d,
\mu)$, the corresponding invariant measure $\mu$ given by Prop.
\ref{prop:proposition1}, 2, is the stationary measure for the
Ornstein-Uhlenbeck process in $\R^d$.

Let us now derive corresponding results for an operator defined on
the Schwartz space of test functions $S(\R^d)$ by
\begin{equation}
 L^{(\beta)} = \Delta + \beta(x) \cdot \nabla \:,\: D(L^{(\beta)}) = S(\R^d)\:.
\end{equation}
We assume that $\beta(x) \cdot \nabla f $ is well defined for all $f
\in S(\R^d)$. Note that $L^{(\beta)} =L_0$, with $L_0$ as in
\ref{prop:proposition1}, if $\beta(x)=-x$. We have the following

\begin{prop} \label{proposition2}
   \begin{itemize}
    \item [(i)] If $\beta$ is such that both $\beta(\cdot) \nabla f$ and $(\nabla\beta) \cdot f$ are well defined in $L^2(\R^d)$, for all $f \in S(\R^d)$, then  the adjoint of $L^{(\beta)}$  (looked upon as an operator) in $L^2(\R^d)$ is given by
    \begin{equation}\Delta -\beta(x) \cdot \nabla
    -(\nabla\,\beta(x))
    \:,
    \end{equation}
    where $\nabla (\beta(x))=\,div\,\beta(x)$ is the divergence of
    $\beta(x)$ (first defined in the distributional sense, but such that $\nabla \beta$ maps $S(\R^d)$ into $L^2(\R^d)$).
    \item[(ii)] Assume that there exists  $G: \R^d \rightarrow \R$, such that $\beta(x)=-\nabla G(x)$, in the distributional sense, and $e^{-G} \in L^1(\R^d)$. Assume the terms $ \nabla G \cdot \rho^{(\beta)}$ and $\Delta G \cdot \rho^{(\beta)}$ are in $L^1(\R^d, f \dd x)$, for any $f \in S(\R^d)$. Then:
\begin{equation}
     \mu^{(\beta)}(\dd x)= \rho^{(\beta)} (x) \dd x\:, \: where \: \rho^{(\beta)} (x) = e^{-G(x)} \:, \: is \: L^{(\beta)}-invariant\:.
\end{equation}
    \item[(iii)] The adjoint of $\left( L^{(\beta)}, D(L^{(\beta)})\right)$ in $L^2(\R^d,\mu^{(\beta)})$ is equal to $L^{(\beta)}$ on $D(L^{(\beta)})$,
    hence $L^{(\beta)}$ is symmetric as an operator in $L^2(\R^d,\mu^{(\beta)})$.
\item[(iv)] If $\beta$ satisfies the assumptions such that the Schr$\ddot{o}$dinger operator $-\Delta +V(x)$ with $V(x) = \beta^2(x)+div \beta(x)$,
 is essentially self-adjoint in $L^2(\R^d)$, on $\mathcal{S}(\R^d)$, then the closure $\overline{L^{(\beta)}}$ with domain  $D(\overline{L^{(\beta)}})$ of
 $\left( L^{(\beta)}, D(L^{(\beta)})\right)$ is self-adjoint in $L^2\left( \R^d,\mu^{(\beta)}\right)$.
\item[(v)] $\mu^{(\beta)}$ is invariant under the one-parameter strongly continuous semigroup $e^{t \overline{L^{(\beta)}}}, t\geq0$, in $L^2(\R^d,\,\mu^{(\beta)}).$
    \end{itemize}
    \end{prop}
\begin{proof}
   The proof is entirely similar to the one of Proposition \ref{prop:proposition1}.
    \begin{itemize}
     \item[(i)] For any $f,g \in \mathcal{S} (\R^d)$ we have
    \begin{eqnarray} \label{eq:star}
        \int L^{(\beta)} f(x) g(x) \dd x &=& \int \left( \Delta + \beta(x) \nabla \right) f(x) g(x) \dd x\,\nonumber\\&=&
            \int f(x) \Delta g(x) \dd x - \int f(x) \nabla(\beta(x) g(x)) \dd x \nonumber\\&=&
       \int f(x) \Delta g(x) \dd x - \int f(x) \left(\nabla \beta(x) \right) g(x) \dd x \\&-& \int f(x) \beta(x) \nabla g(x) \dd x
       \:\nonumber
    \end{eqnarray}
    \item[(ii)] Let us take $g= \rho^{(\beta)}$ in \eqref{eq:star}, then we get
    \begin{eqnarray}\label{eq:squarecrossed}
  \dint L^{(\beta)} f d\mu^{(\beta)} &=& \dint \left( L^{(\beta)} f\right) (x)\, \rho^{(\beta)}(x)\,\dd x \nonumber\\&=&\dint f(x) \Delta\rho^{(\beta)}(x) \dd x -
  \dint f(x) \left( \nabla \beta\right)(x) \rho^{(\beta)} (x) \dd x \nonumber\\&
  - &\dint f(x) \beta(x) \nabla \rho^{(\beta)} (x)\dd x \:.
    \end{eqnarray}
    But $\nabla \rho^{(\beta)}(x) = -\nabla G(x) \,
    \rho^{(\beta)}(x)$, by definition of $\rho^{(\beta)}$.\\

    Moreover $\Delta \rho^{(\beta)}(x) = \nabla G(x)^2 \rho^{(\beta)}(x) - \Delta G(x) \rho^{(\beta)}(x)$. Introducing this into
    \eqref{eq:squarecrossed} we get, using $\beta= -\nabla G$:
        \begin{eqnarray*}
      \dint L^{(\beta)} f(x) \,\rho^{(\beta)}(x)\,\dd x  &=& \dint f(x) \Delta \rho^{(\beta)}(x) \dd x +
        \dint f(x) G(x) \rho^{(\beta)}(x) \dd x\,\nonumber\\ &-& \dint f(x) \left( \nabla G\right)^2 (x) \rho^{(\beta)}(x) \dd x
        \nonumber\\&=& \dint f(x) \left( \nabla G\right)^2 (x) \rho^{(\beta)} (x) \dd x - \dint f(x) (\Delta G) (x) \rho^{(\beta)}(x) \dd x \nonumber\\&+&
       \dint f(x) \Delta G (x) \rho^{(\beta)} (x) \dd x - \dint f(x) \left( \nabla G\right)^2(x) \rho^{(\beta)} (x) \dd x \nonumber\\&=&0
    \end{eqnarray*}

 \item[(iii)] We repeat the steps of proof of the corresponding statement in \eqref{prop:proposition1}.


           We have, for any $f, g \in  \mathcal{S} (\R^d)$, using \eqref{eq:star} with $g$ replaced by $g \rho^\beta$
            \begin{eqnarray}
      \dint L^{(\beta)} f(x) \left(g \rho^{(\beta)}\right) (x)\,\dd x  &=& \dint f(x) \Delta \left(g \rho^{(\beta)}\right) (x) \dd x -
        \dint f(x) \left( \nabla \beta(x) \right) \left( g \rho^{(\beta)}\right) (x) \dd x\,\nonumber\\
                &-& \dint f(x) \beta(x) \nabla\left( g \rho^{(\beta)}\right) (x) \dd x
        \nonumber\\&=&
                \dint f(x) \left( \Delta g\right) (x) \rho^{(\beta)} (x) \dd x + \dint f(x) 2 \left( \nabla g \right) (x) \nabla \rho^{(\beta)} (x)        \dd x \nonumber\\&+&
       \dint f(x) g(x) \Delta  \rho^{(\beta)} (x) \dd x - \dint f(x) \left( \nabla \beta(x) \left( g \rho^{(\beta)} \right) \right) (x)   \dd x \nonumber\\&-&
            \dint f(x) \beta(x) \nabla g(x) \rho^{(\beta)}(x) \dd x \nonumber\\&-&
            \dint f(x) \beta(x) g(x) \nabla \rho^{(\beta)}(x) \dd x,
    \end{eqnarray}
        which is the analogue of \eqref{eqn:equation5}. Inserting the formula for $\nabla \rho^{(\beta)}$, resp. $\Delta \rho^{(\beta)}$ after \eqref{eq:squarecrossed}, into the latter formula we get
        \begin{eqnarray}
      \dint L^{(\beta)} f(x) \left(g \rho^{(\beta)}\right) (x)\,\dd x  &=& \dint f(x) \Delta g(x)  \rho^{(\beta)} (x) \dd x
            \nonumber\\&-&
        2 \dint f(x)  \nabla g(x)  \nabla G(x)  \rho^{(\beta)}  (x)  \dd x
                \nonumber\\&+&
                \dint f(x) g(x) \nabla G(x)^2 \rho^{(\beta)} (x) \dd x \\
                &-&
                \dint f(x)  \nabla g (x) \Delta G(x) \rho^{(\beta)}(x) \dd x \nonumber\\&-&
                \dint f(x) \nabla \beta(x) g(x) \rho^{(\beta)}(x) \dd x \nonumber\\&-&
                \dint f(x) \beta(x) \nabla g(x) \rho^{(\beta)}(x) \dd x  \nonumber\\&+&
                \dint f(x) \beta(x) g(x) \nabla G(x) \rho^{(\beta)}(x) \dd x \quad. \nonumber
    \end{eqnarray}
        Using $\beta(x) = \-\nabla G(x)$ , $\nabla \beta(x) = -\Delta G(x)$, we see that the second term plus the last but 1 term yield $1/2$ of the second term,
        the $3$ term cancels with the last one, the last but 2 term cancels with the 4 term and we remain with
     \begin{equation}
     \dint f(x) \Delta g(x) \rho^{(\beta)}(x) \dd x - \dint  f(x) \nabla g(x) \nabla G \rho^{(\beta)}(x) \dd x \:,
     \end{equation}
      which yields the claimed result.


\item[(iv)]This is similar as for (iv) in {Prop.1}, the ``ground state transformation'' is obtained replacing $\mu$ by $\mu^{(\beta)}$ and $\rho$ by
$\rho^{(\beta)}$, then
\begin{equation}\label{eqn:groundstate}
 U^{-1} \left( \Delta + \beta \cdot \nabla \right) U f = \left( \Delta - (\nabla \beta)^2 - \nabla \beta  \right) f \:.
\end{equation}
 Under our assumptions on $\beta$ the operator on the right hand side of the \eqref{eqn:groundstate}, which is of the
Schr$\ddot{o}$dinger type, with $V(x) = \nabla \beta(x)^2 + \nabla
\beta(x)$ ,  is essentially self-adjoint in $L^2(\R^d)$, hence its
closure is self-adjoint.
\item[(v)] This is entirely similar to the proof of the corresponding statement in Proposition \eqref{prop:proposition1}.
    \end{itemize}
\end{proof}

\begin{rem}
  For examples where the assumptions on $\beta$ in (iv) of Proposition \ref{proposition2} are satisfied see, e.g., \cite{AHKS},
  \cite{RS}.
\end{rem}
The following corollary is immediate:
\begin{cor}
 If $\beta(x) = -x + F(x),\,x\,\in\,\R^d$, so that $G(x) = \frac{x^2}{2} + G_F (x)$, with $\nabla G_F (x) = -F(x)$, then $\rho^{(\beta)}(x) = e^{-G_F(x)}\rho(x)$,
  with $\rho$ as in Proposition \ref{prop:proposition1}.
\end{cor}

Let us now apply similar ideas to the case of operators of the form
\begin{equation}\label{eqn:L0L1}
 \left( L_0 f \right) (x) = \beta (x) \nabla f(x) + L_1 f(x) \:,
\end{equation}
where $L_1$ is a pseudodifferential operator and $f,g \in
\mathcal{S}(\R^d)$. On $\beta$ we assume that it has a Fourier
transform in the distributional sense. Then
    \begin{equation*}
       \widehat{L_0 f}(k) = i \dint \widehat{\beta} (k-q) q \widehat{f}(q) dq + \widehat{L_1 f} (k) \:,
    \end{equation*}
    where $\widehat{}$ stands as before for Fourier transform, s.t. $\widehat{\nabla f} (k) = ik \widehat{f} (k)$.
    Suppose first for simplicity that $\widehat{L_1 f}(k)= M(k) \widehat{f}(k)$, where $k \in \R^d$ , for some measurable function $M$ (e.g. $L_1$ of the form
    of the term with coefficient $a_2$ in \eqref{eqn:levyoperator}. Then the adjoint of $M$   in $L^2(\R^d,\dd k)$ is $M$ itself and hence, for any
    $g \in \mathcal{S}(\R^d)$, we have
  \begin{equation}
     \int \left( M \widehat{f}\right) (k) \widehat{g}(k) \dd k = \int \widehat{f}(k) \left( M \widehat{g}\right) (k) \dd k \:.
  \end{equation}
    Moreover
\begin{equation}
    \int \beta(x) \left( \nabla f (x) \right) g(x) \dd x = -\int f(x)  \nabla \left( \beta g \right) (x) \dd x = - \int \overline{\widehat{f}} (k) i k  \widehat{\beta g} (k) \dd k
    \:,
 \end{equation}
 where in the last equality we used Parseval formula.\\
    Hence
  \begin{equation}
     \int L_0 f(x) g(x) \dd x = \int \overline{\widehat{f}}(k) \left( M(k)\widehat{g}(k) - i k  \, \widehat{\beta\,
     g} \right) (k) \dd k \:.
 \end{equation}
    From this we deduce that the adjoint of $\left( L_0 ,  ( \mathcal{S} (\R^d) \right)$ in $L^2(\R^d)$ is the inverse Fourier transform of the operator
    $g(k) \rightarrow M(k) g(k) - i\, k\,\int \beta (k-q) g(q) dq$\, in\, $L^2 (\R^d , \dd k )$.
    Hence, setting $g(x)\,dx =\mu(dx) $ we  find that $\mu$ is
    $L_0-$invariant if
   \begin{equation}
     \int L_0 f(x) \mu (\dd x) = \int \overline{\widehat{f}} (k) M(k) \left( \widehat{\mu}\right) (\dd k) - i  \int k\,\overline{\widehat{f}}(k) \widehat{\beta} (k-q) \widehat{\mu}(dq) =0 \:,\:
     \forall f \in \mathcal{S}(\R^d) \:.
 \end{equation}
    This yields a linear  equation for the probability measure $\mu$ which involves convolution
        \begin{equation}\label{eqn:star2}
     -i \left( M(k) \left( \widehat{\mu} \right)   \right) (\dd k) =  \left( k \widehat{\beta} \star \widehat{\mu} \right) (k) \:,
        k \in \R^d  \backslash \left\{ 0 \right\} \:,
    \end{equation}
    as distributions in $\mathcal{S}^\prime(\R^d)$, provided of course both sides can be interpreted as such distributions.

\begin{rem}

\begin{itemize}
\item[]
         \item [{\rm (1)}]The existence of solutions of \eqref{eqn:star2} depends on the multiplicative operator $M(k)$, and on the convolution kernel $\widehat{\beta} (k-q)$, $k, q \in \R^d$ .
                E.g. if $\beta(x) = -x$, $M(k) = a_2 C_\alpha k^\alpha$, $0 < \alpha \leq 2$, one solution of \eqref{eqn:star2} is given by $\mu(\dd x) = \rho_2(x) \dd x$,
         with $\rho_2$ as in Example 2.14.
                \item[{\rm (2)}] Equation \eqref{eqn:star2} can be looked upon as an homogeneous linear equation $A_k \,\widehat{\mu}(k)=0$, where
                $A_k := - i M(k) + k \widehat{\beta} \star \:, k \in \R^d \backslash \left\{ 0 \right\} \:,$ acting on the Fourier transform  $\widehat{\mu}$ of positive
                measures $\mu$. For $d=1$ this is a homogeneous linear convolution equation with non constant coefficients. Thus we  have only solutions if
                $A_k$ has a non trivial kernel.
                \end{itemize}
\end{rem}

\subsection{Probabilistic methods to identify the associated stochastic differential equations}\label{Section2.3}
Let $(X(t))_{t\geq 0}$ be the solution of the following stochastic
differential equation:
\begin{equation}\label{eq:SDEsymb}
\begin{aligned}
    \dd X(t)&= \Psi(X(t))\dd t + \Phi(X(t)) \dd L(t),\\
    X(0)&=x;
\end{aligned}
\end{equation}
where $\Psi, \Phi$ are globally Lipschiz continuous mappings,
respectively from $\R^d$ into itself and into the space of symmetric
positive definite matrices, while $(L(t))_{t\geq 0}$ is a
$d$-dimensional L\'evy process with generating triplet $(Q,N,\ell)$,
(see Sect. 2 for this terminology).\\ Existence and uniqueness of a
strong solution to this equation are known, see,e.g.,
\cite{GihSko,MaRuBook}, and $(X(t))_{t\geq 0}$ is a time-homogenous
Markov process. As usual we can associate to $(X(t))_{t\geq 0}$ a
semigroup $(P_t)_{t\geq 0}$ of operators on $B_b(\R^d)$ by setting
\begin{align*}
    P_t u(x):= \mathbb{E}^x u(X(t)), \qquad t\geq 0, x\in \R^d,\,u\,\in\,B_b(\R^d).
\end{align*}
This semigroup is Markov and conservative (i.e. $P_t {\bf 1} = {\bf
1}$), and Feller, i.e. $P_t$ leaves invariant  $C_0(\R^d)$ (the
space of continuous functions on $\R^d$, which vanish at infinity)
and
\begin{align*}
   &\lim_{t\to 0} \| P_tu-u\|_{\infty} =0, \qquad for \ every \ u\in C_0(\R^d) \:,\
   \|\cdot \|_{u} \ \textrm{being the sup-norm}
\end{align*}
see, e.g., \cite{APD}. To $P_t$ corresponds the infinitesimal
generator $(A,D(A))$ which is defined by
\begin{align}\label{eq:lim}
    Au:=\lim_{t \to 0} \frac{P_t u-u}{t}
\end{align}
with the domain consisting of all $u\in C_0(\R^d)$ for which the
limit \eqref{eq:lim} exists.

A classical result due to Courr$\grave{e}$ge, see \cite{Cou} or
\cite{APD}, Th.3.5.3, p.158 and Th. 3.5.5, p.159,  shows that, if in
addition to the previous assumptions,  $C^\infty_c:=C^\infty_c(\R^d)
\subset D(A)$, then $A|_{C^\infty_c}$ is a pseudo differential
operator with symbol $-p(x,\xi)$, i.e. $A$ can be written as
\begin{align}\label{eq:AFour}
     Au(x):= -\int_{\R^d} e^{i\langle x, \xi\rangle} p(x,\xi) \hat{u}(\xi) \dd \xi, \qquad u\in C^\infty_c \:,\:  x\in \mathbb{R}^d
\end{align}
where $\langle \cdot, \cdot \rangle$ is the scalar product in
$\R^d$, $\hat{u}$ denotes the Fourier transform $\hat{u}(\xi) =
\frac{1}{(2\pi)^d} \int e^{- i \langle x, \xi\rangle} f(x)
dx,\,\xi\,\in\,\R^d$ and $p: \R^d \times \R^d \to \mathbb{C}$ is
locally bounded and, for fixed $x$, a continuous negative definite
function in the sense of Schoenberg in the co-variable $\xi$ (we
denote by $C^{\infty}_c(\mathbb{R}^d)$ the space of smooth
continuous real-valued functions on $\mathbb{R}^d$ with compact
support). This means that $p(x,\xi)$ admits a L\'evy-Khintchine
representation
\begin{align}\label{eq:symb}
   p(x,\xi)= -i \langle\ell(x),\xi\rangle+\frac{1}{2}\langle\xi Q(x) ,\xi\rangle- \int_{y\neq 0} \left(
   e^{i\langle\xi, y\rangle}-1-i\langle\xi, y\rangle \, {\bf 1}_{B_1} (y)\right) N(x,\dd
   y), x,\,\xi \,\in\,\R^d.
\end{align}
For each $x\in \R^d$ $(Q(x),N(x,\dd y),\ell(x))$ is a L\'evy triplet
in the sense of Sect. \ref{Section2.1} (depending parametrically on
$x\,\in\,\R^d$). The function $p(x,\xi)$ is called the symbol of the
operator and $N(x,dy)$ will be called the L\'evy kernel. Notice that
the killing term is absent due to the conservativeness of $P_t$.
Alternatively, using Remark (\ref{BAABA}) we can replace the term
containing ${\bf 1}_{B_1} (y)$ by ${1 \over {1+\mid y
\mid^2}},\,y\,\in\,\R^d,$ by simultaneously changing the drift term
by changing $ \ell(x)$ to $\ell'(x)=\ell(x)+\,\int_{\R^d}\,({{1
\over {1+\mid y\mid^2}}-{\bf 1}_{B_1(0)} (y)})\,N(x,\,dy).$ For
details we refer to, e.g., Jacob \cite[Chapter 45, pgg.
342-364]{Jac}. Combining \eqref{eq:AFour} and \eqref{eq:symb} the
generator $A$ of a Feller process satisfying the condition
$C^\infty_c \subset D(A)$ can be written in the following way:
\begin{equation}\label{eq:op-non-loc}
\begin{aligned}
     A u(x) = \langle\ell(x), \nabla u(x)\rangle + \frac{1}{2} {\rm Tr}[\sqrt{Q}(x) \nabla^2 u(x) \sqrt{Q}^*(x)]
    \\+
    \int_{y\neq 0} \left( u(x+y)-u(x)- \langle y , \nabla u(x){\bf 1}_{{\bf
1}_{B_1(0)}}(y)\rangle\right) N(x,\dd y),\,x\,\in\,\R^d,
\end{aligned}
\end{equation}
for all $u\in C^\infty_c(\R^d)$ ($*$ standing for the adjoint of
matrices in $\R^d$). Thus from the symbol we obtain the
integro-differential form of the infinitesimal generator of the
process.

\begin{rem}\label{yomma}
    We recall that every L\'evy process $(L(t))_{t\geq 0}$ with triplet $(Q,N,\ell)$ (in the sense of section $2.1$ and \cite[p.65]{Sato}) on $\mathbb{R}^d$ has the following
L\'evy-Ito decomposition
\begin{align}\label{eq:rem}
   L(t)= \ell t + \sqrt{Q} \dd W(t) +  \int_{B_1} y \left( \mu^L([0,t],\dd y) - t N(\dd y) \right) + \int_{B_1^c} y \mu^L([0,t],\dd y),
\end{align}
where $\mu^L$ is the Poisson point random measure given by the jumps
of $L$ whose intensity measure is the L\'evy measure $N$, (with
$B_1^c:=\R^d-B_1$). This means that, for any $ B \in {\cal
B}(\R^d),\,
  \mu^L([0,t],B)(\omega)=\int_{B} \mu^L([0,t],\dd y) (\omega)
$ is the number of $s\in [0,t]$ with $L_s(\omega)-L_{s^-}(\omega)
\in B$ for $\omega \in \Omega$ (the set of c\`{a}dl\`{a}g paths of
$L$). One has $\mu^L([0,t],B) = t \mu([0,1],B)$, and
$\mu^L([0,1],B)$ has Poisson distribution with mean $N(B)$ (see
\cite[p.119]{Sato}, \cite[p. 87]{APD}, \cite{RozMi}). The last term
in \eqref{eq:rem} can also be written as
\begin{align*}
   \sum_{0<s\leq t} \Delta L(s) {\bf 1}_{|\Delta(s)|\geq 1}.
\end{align*}
It turns out that the infinitesimal generator of $L(t)$ is given by
\begin{multline*}
    A u(x) = \langle\ell,\nabla u(x) \rangle+ \frac{1}{2} \sqrt{Q} \nabla^2 u(x) \sqrt{Q}^*
   \\ +
    \int_{y\neq 0} \left( u(x+y)-u(x)- \langle y,\nabla u(x) \rangle {\bf 1}_{B_1}(y)\right) N(\dd y),\,x\,\in\,\R^d,
\end{multline*}
which is well-defined on $C^\infty_c(\R^d)$. Hence, following the
arguments above, we see that the symbol of this $A$ coincide with
the characteristic exponent of the $L(t)$, i.e. L\'evy processes are
exactly those Feller processes whose generator has constant
coefficients and $p(x,\xi)\equiv \psi(\xi),\,\xi\,\in\,\R^d$, where
$\psi$ is the function introduced in Proposition \ref{thm:satz1}.
\end{rem}

We are interested in determining the symbol of the process
$(X(t))_{t\geq 0}$ corresponding with equation \eqref{eq:SDEsymb},
since it allows us to determine the integro-differential form of the
infinitesimal generator of the process. This is a key point in
finding the expression of the invariant measure corresponding with
$(X(t))_{t\geq 0}$ (see Subsection \ref{Section2.5}). In
\cite{Jacob-Sch}
it is proven that
 the symbol $-p(x,\,\xi)$ of $A$ coincides with minus the symbol of the process, which is defined by
\begin{align*}
    p(x,\xi):= -\lim_{t\to 0} \mathbb{E}^x \frac{e^{i\langle
    (X^\sigma(t)-x),\xi\rangle}-1}{t},\,x, \xi\,\in \R^d,
\end{align*}
where $\sigma =\sigma^{x,R}$ is the first exit time of $X(t)$,
started at $x$, from the ball of radius $R>0$. \\The notation
$X^\sigma(t)$ stays for the process $X(t)$, started at $x$, and
stopped at time $t\geq 0$ when it exist from the ball of radius $R$.
In particular, in the case of $(X(t))_{t \geq 0}$ being the solution
of equation \eqref{eq:SDEsymb} we have, see \cite{Cou, Jac, Jac0,
RZI}, that
\begin{align*}
    p(x,\xi) = \psi(\Phi(x) \xi) - i \langle \Psi(x),\xi \rangle,
\end{align*}
where $\psi$ is the characteristic exponent of $(L(t))_{t\geq 0}$
and $\Psi(x)$ is the first coefficient (``drift coefficient'') in
\eqref{eq:SDEsymb}. Thus we have (with $( Q, N, l)$) as in Remark
\ref{yomma}.)
\begin{multline*}
    p(x,\xi)= i  \langle \ell, \Phi(x) \xi\rangle -\frac{1}{2}\langle \Phi(x)\xi, Q  \Phi(x)\xi\rangle
   + \\
  \int_{\R^d} \left(e^{i \langle y,\Phi(x) \xi\rangle}- 1- i \langle y, \Phi(x)\xi \rangle {\bf 1}_{B_1} (y) \right) N(\dd y) - i \langle \Psi(x),\xi \rangle
\end{multline*}
where $\Phi$ is the second coefficient in \eqref{eq:SDEsymb}. The
term containing the integral can be equivalently written as
\begin{equation}\label{eq:symb-rev}
  \int_{\R^d} \left(e^{i \langle  \xi,\tilde{y}\rangle}- 1- i \langle \xi, \tilde{y}\rangle {\bf 1}_{B_1} (\Psi^{-1}_x\tilde{y}) \right)\tilde{N}(x,\dd \tilde{y}),
\end{equation}
where $\tilde{N}(x,\dd \tilde{y})$ is the image measure of $N(\dd
y)$ under the transformation $y\,\in\,\R^d \mapsto \tilde{y}:= \Psi_x(y) =
\Phi(x) y$, $x,y \in \mathbb{R}^d$ ( this can be seen by taking
Fourier transforms). Now comparing expression \eqref{eq:symb-rev} with
\eqref{eq:symb}, we see that the integro-differential  operator
corresponding with the solution of the stochastic differential
equation \eqref{eq:SDEsymb} is given by, (cf. \cite{Ku}):
\begin{multline}\label{eq:Imk-Kurtz}
    A u(x) = \langle \ell \Phi^\star (x)- \Psi(x), \nabla u(x)\rangle
   + \frac{1}{2}{\rm Tr} [\sqrt{\tilde{\Phi}(x)} \nabla^2 u(x) \sqrt{\tilde{\Phi}^\star(x)}]  \\
   + \int_{\R^d} (u(x+\tilde{y})-u(x) -\langle \tilde{y},\nabla u(x)\rangle{\bf 1}_{B_1}(\Psi_x^{-1}\tilde{y})) \tilde{N}(x,\dd
   \tilde{y}).
\end{multline}
By considering the inverse transformation
$\Psi^{-1}_x(\tilde{y})=\Phi^{-1}(x)\tilde{y}=y$, we get
\begin{multline*}
    A u(x) = \langle \ell \Phi^\star (x)- \Psi(x), \nabla u(x)\rangle
   + \frac{1}{2}{\rm Tr} [\sqrt{\tilde{\Phi}(x)} \nabla^2 u(x) \sqrt{\tilde{\Phi}^\star(x)}]  \\
   + \int_{\R^d} (u(x+\Phi(x)y)-u(x) -\langle \Phi(x)y,\nabla u(x)\rangle{\bf 1}_{B_1}(y)) N(\dd y),
\end{multline*}
since, by construction, $\tilde{N}$ is the image measure of $N$
under $\Psi_x$. Again the factor ${\bf 1}_{{B_1}}(y)$ can be
replaced in all formulae by ${1 \over {1+\mid y\mid^2}},$ by
changing correspondingly $\ell \Phi^\star (x)- \Psi(x)$ by $\ell
\Phi^\star (x)- \Psi(x)+\dint_{\R^d}\,({1 \over {1+\mid
y\mid^2}}-{\bf 1}_{{B_1}}(y))\,N(x,dy).$\\ The latter representation
coincides with the representation given e.g. in \cite{APD} (p.341).

\begin{rem}
   Comparing \eqref{eq:Imk-Kurtz} with the pseudo-differential operators given in  \cite[(2.33) and (2.37) pag. 13]{IMKN}, we see that all expressions coincide.

In the case where  $(L(t))_{t\geq 0}$ is a pure jump process (i.e.
$(Q,N, \ell)=(0,N,0)$, the expression  for $Au(x)$ can be further
simplified; we obtain
\begin{align*}
   Au(x) = \langle \Psi(x), \nabla u(x)\rangle
   + \int_{\R^d} (u(x+y)-u(x) -\langle y,\nabla u(x)\rangle{\bf 1}_{|\Phi(x)y|<1}(y)) \tilde{N}(x,\dd y).
\end{align*}
Moreover, we notice that, by the definition of $\tilde{N}(x,\dd \tilde{y})$
we have, for any $\Gamma\,\in\,B(\R^d):$
\begin{align*}
    \tilde{N}(x,\Gamma)&= N(\Phi(x)\Gamma)= \int_{\R^d} {\bf 1}_{\Phi(x)\Gamma}(\tilde{y}) N(\dd \tilde{y}) \\
     &= \int_{\R^d} {\bf 1}_{\Gamma}(\Phi(x)^{-1}\tilde{y}) N(\dd \tilde{y}).
\end{align*}
We notice that the representation above is the same representation as given in \cite[p.119]{Ku},
with  $\lambda(x,\tilde{y})=1$ and $\gamma(x,y)= \Phi^{-1}(x)y$ in \cite{Ku}.
\end{rem}
\subsection{The inverse problem: invariant measures via ground state transformations} \label{Section2.4}
By the  considerations in Sections \ref{Section2.1} and
\ref{Section2.2} we have, in particular, {\it concrete}
invariant measures for process of the form $\dd X(t)= A X(t) \dd t+
\dd L(t) $, with $A= -Q$ as in proposition \eqref{PropositionA}. We
shall now see that by extending the type of ``ground state
transformation'' (Doob-h-transform), similar to the ones one
performs in the case of processes satisfying equations of the form
\begin{equation}\label{eq:moon}
 dX(t) = A X(t) \dd t+ \beta(X(t)) \dd t+ \dd L(t),
\end{equation}
with $L(t)$ of the Gaussian type, $\beta$ of gradient type, one can find explicit invariant
measures also for equations of the form \eqref{eq:moon} for general
L\'evy noise. This provides an alternative somewhat complementary
procedure to the one we discussed in Sec. \ref{Section2.2}. For
this extension we follow closely \cite{A-Cufaro}, who were the
first, to the best of our knowledge, who extended previous work on
the ground state transformation for the case with Gaussian noise
covered in \cite{AHKS} to the case of L\'evy noise.\\
Let $\phi$ be a given function on $\R^d$, such that $\int_{\R^d}
\phi^2 \dd x = 1$ and $\phi(x)>0,$ $\dd x-a.e.$
Let $\mu (\dd x) = \phi^2(x)\dd x$. Define $H$, for any $f \in
C_0^\infty (\R^d)$, as an operator acting in $L^2(\R^d,\,dx),$ by
\begin{equation} \label{eq:H}
 \left( H f \right)(x) = - \frac{L_0 (\phi f )-fL_0(\phi))}{\phi} (x)\:,
\end{equation}
for all  $x$ s.t. $\phi(x) > 0$, where $\left( L_0 , D(L_0) \right)$
is the infinitesimal generator acting in $L^2(\R^d,\,dx),$ of a $\dd
x$ symmetric L\'evy process $Z_t$ taking values in $\R^d$ (this
means that the law $P_{Z_t}$ of $Z_t$ is symmetric under reflection
$y\,\longrightarrow\, {-y}$ in $\R^d$, cf. \cite[pag. 153]{APD}. We
shall see below that the right hand side of \eqref{eq:H} is well
defined even without assuming $\phi f \in D(L_0)$. Let us recall
that a $\dd x-$symmetric L\'evy process has a generator which is self-adjoint in $L^2(\R^d, \dd x)$, (or,
equivalently, the associated Dirichlet form is symmetric in
$L^2(\R^d,\dd
x)$), (see, e.g., {\cite{FuOTa}, \cite{MaRo}, \cite{A})}.\\
$L_0$ is thus of the form of $\mathcal{L}^L$ as given by (\ref{bbbb}) but
with the restriction of its being symmetric in
$L^2(\mathbb{R}^d,\mathrm{d}x)$, which forces the choice $\gamma=0$
and the absence of the term containing the gradient in the integral,
i.e.\ $L_0$ is of the form $L_0 = L_{0,G} + L_{0,J}$, with
\begin{eqnarray}\label{dddd}
    (L_{0,G}f)(x) &= \frac{1}{2} \sum_{j,k=1}^{d} q_{jk} \frac{\partial}{\partial x_i \partial_k} f(x) \nonumber\\
    (L_{0,J}f)(x) &= \int\limits_{\mathbb{R}^d} [f(x+y) - f(x)] \nu (\mathrm{d}y) \text{,}
\end{eqnarray}
with $\nu (\mathrm{d}y) = \nu (-\mathrm{d}y)$, $f \in D(L_{0,G})
\cap D(L_J) \subset D(L_0)$. Note that we still have, for (\ref{bbbb}),  $D(L_0) \supset
C_0^{\infty}(\mathbb{R}^d)$. This by (\ref{eq:symb-rev}), (\ref{eq:Imk-Kurtz}) corresponds to having the symbol
associated with $L$ as follows
\begin{equation*}
    p(x,\xi)=\eta(\xi) = -\frac{1}{2} \langle Q\xi,Q\xi \rangle + \dint (\cos\langle \xi,y \rangle - 1) \nu (\mathrm{d}y) \text{,}
\end{equation*}
$\xi \in \mathbb{R}^d$, independent of $x \in \mathbb{R}^d$. \\
The (symmetric, positive) pre-Dirichlet form $\mathcal{E}_{L_0}^{0}$
in $L^2(\mathbb{R}^d,\mathrm{d}x)$ associated with $L_0$ is:
\begin{align*}
     \quad \mathcal{E}_{L_0}^{0} (f,g) &= (-L_0f,g)_{L^2(\mathbb{R}^d,\mathrm{d}x)}. \\
   \text{Hence } \mathcal{E}_{L_0}^{0} (f,g) &= \mathcal{E}_{G}^{0} (f,g) + \mathcal{E}_{J}^{0} (f,g) \:.
\end{align*}
We have with
\begin{align*}
    \mathcal{E}_{G}^{0} (f,g) &:= \frac{1}{2} \int \nabla f(x) \cdot Q \nabla g(x) \, \mathrm{d}x \\
    \mathcal{E}_{J}^{0} (f,g) &:= \frac{1}{2} \int\limits_{\mathbb{R}^d} \int\limits_{\mathbb{R}^d \backslash \left\{0\right\}}[f(x+y) - f(x)] [g(x+y) - g(x)] \nu (\mathrm{d}y),
\end{align*}
as a simple computation shows (integration by parts, for the term with derivative, change of variables and exploitation of reflection symmetry of $\nu$, for the other term) (cfr.\cite[pag. 166]{APD}). We observe that $\mathcal{E}_{G}^{0} (f,g)$ can aso be written in the form
$$\mathcal{E}_{G}^{0} (f,g) = \frac{1}{2} \int\limits_{\mathbb{R}^d} \int\limits_{\mathbb{R}^d\backslash D} \left[ f(x) - f(y) \right]
\left[ g(x) - g(y) \right] J(dx,dy) \;,$$
where $J(dx,dy) := \frac{1}{2} \left[ \nu_x(dy)dx + \nu_y(dx) dy\right]$, and $\frac{1}{2}\nu_x(B):=\nu(B-x)$, $x \in \mathbb{R}^d$, $B \in \mathcal{B}(\mathbb{R}^d)$, $D:= \left\{ (x,y) \in \mathbb{R}^d \times \mathbb{R}^d , x \neq y\right\}$.
\\
Under suitable assumptions on $\nu,$ see \cite{ASK, AlRo}, $\mathcal{E}_L^0$ is closable and taking the closure $\mathcal{E}_L$ we have a natural minimal Dirichlet form in $L^2(\mathbb{R}^d,\mathrm{d}x)$ associated with a closed extension of $(L_0,D(L_0))$ in $L^2(\mathbb{R}^d,\mathrm{d}x)$. \\
Now let us assume $\phi \in H^{1,2}(\mathbb{R}^d, dx)$, $\phi(x)>0$, for all $x \in \mathbb{R}^d$, and consider on $C_0^{\infty}(\mathbb{R}^d)$
\begin{equation}\label{circle}
  -H_G=  L_{0,G} + \beta(x) \cdot \nabla
\end{equation}
where $\beta(x) = \nabla \ln \phi(x)$.
We can look upon $H_G$ as an operator acting on  $C_0^{\infty}(\mathbb{R}^d)$ in $L^2(\mathbb{R}^d,\mu)$, with $\mu(\mathrm{d}x) = \phi(x)^2 \, \mathrm{d}x$, as before see, e.g.,  \cite{ASK}  \\
It is symmetric on this domain and negative definite, as seen by integration by parts (see \cite{ASK}). In fact
$$
 \left( f, H_G g\right)_{L^2\left(\mathbb{R}^d, \mu \right)}=\frac{1}{2} \int\limits_{\mathbb{R}^d} \nabla f \cdot Q \;\nabla g  d\mu \:.
$$
To it there is associated the classical pre-Dirichlet form $(f, H_G g)_{L^2(\mu)} = \int \nabla f \cdot Q \nabla g d\mu$, $\mathbb{R}^d, f,g \in C_0^\infty \left(\mathbb{R}^d\right)$, as also seen by integrating by parts.
Let us now consider $-H$ as an operator in $L^2(\mathbb{R}^d,\mu)$, defined by
 \begin{equation} \label{gggg} \,-H = -H_G - H_J\,\, with \,\, - H_Jf :=\frac{ L_{0,J}(\phi f) - {-fL_{0,J}\phi}}{\phi}
\end{equation}
Assuming $\phi \in D(L_{0,J})$ and following the computation in the
Appendix of \cite{A-Cufaro} (with our $L_{0,J}$) we get that $\phi f \in
D(L_{0,J})$ and
\begin{equation*}
    L_{0,J}(\phi f) = fL_{0,J}\phi + \phi L_{0,J}f + \int \delta_y \phi \delta_y f \nu (\mathrm{d}y) \text{,}
\end{equation*}
with $(\delta_yf)(x) := f(x+y) - f(x)$. Hence from the definition of
$H_J$, we get, using the expression for $L_{0,J}$, given by (\ref{dddd}) and the definition of $\delta_y:$
\begin{align}\label{star}
    -H_Jg(x) &= L_{0,J}g(x) + \int \frac{\delta_y\phi(x)}{\phi(x)} \delta_yg(x) \nu (\mathrm{d}y) \\
    &= \int [g(x+y) - g(x)] \nu (\mathrm{d}y) + \int \frac{\phi(x+y) - \phi(x)}{\phi(x)} [g(x+y) - g(x)] \nu (\mathrm{d}y) \\
    &= \int [g(x+y) - g(x)] \nu (x;\mathrm{d}y) \text{,}
\end{align}
with $\nu (x;\mathrm{d}y) := \frac{\phi(x+y)}{\phi(x)} \nu (\mathrm{d}y),\,x,\,y\,\in\,\R^d$. \\
It is not difficult to see that $-H_J$ is symmetric in
$L^2(\mathbb{R}^d,\mu)$. In fact define
\begin{equation*}
    \mathcal{E}_{H_J}^0 (f,g) := -(H_Jf,g)_{\mu} \text{,}
\end{equation*}
where $(\, , \,)_{\mu}$ is the scalar product in $L^2(\mathbb{R}^d,\mu)$, $f,g \in C_0^{\infty}(\mathbb{R}^d)$. \\
By the definition of $-H_J$ and the definition of $\mu$ we
have
\begin{align*}
    \mathcal{E}_{H_J}^0 (f,g) &= \int  \frac{L_{0,J} (\phi f) - f L_{0,J} \phi}{\phi}  g \phi^2 \, \mathrm{d}x \\
    &= \int \phi g L_{0,J}(\phi f) \mathrm{d}x  - \int \phi g f L_{0,J} \phi \, \mathrm{dx} \text{.}
\end{align*}
By the definition (\ref{dddd}) of $L_{0,J}$ we then get:
\begin{equation*}
    \mathcal{E}_{H_J}^0 (f,g) = \int \phi g [(\phi f)(x+y) - (\phi f)(x)] \nu (\mathrm{d}y) \mathrm{d}x - \int \phi gf [\phi(x+y) - \phi(x)] \nu (\mathrm{d}y)
        \;, f,g \in C_0^\infty(\mathbb{R}^d)\;.
\end{equation*}
Following \cite{A-Cufaro} or \cite{APD} we see that this can be rewritten
in the symmetric form:
\begin{equation*}
    \mathcal{E}_{H_J}^0 (f,g) = \frac{1}{2} \int (\delta_y f)(x) \delta_y g \phi(x+y) \phi(x) \nu (\mathrm{d}y) \mathrm{d}x \text{.}
\end{equation*}
But this is a symmetric bilinear form, and in fact a jump
pre-Dirichlet form in $L^2(\mathbb{R}^d,\mu)$, i.e.\ is densely
defined, bilinear, positive, closable, under natural assumptions on $\phi$ and $\nu$  (see \cite{AKS}) with jump measure
\begin{equation*}
    J(\mathrm{d}x,\mathrm{d}y) = \frac{1}{2} \{\phi(x+y)\} [ \phi(x)+\phi(y)] \nu (\mathrm{d}y) \mathrm{d}x \text{.}
\end{equation*}
Its closure is then a (positive, symmetric) Dirichlet form
$$\mathcal{E}_{H_J}(f,g)=\frac{1}{2} \int\limits_{\mathbb{R}^d} \int\limits_{\mathbb{R}^d \backslash \left\{ 0 \right\}} (\delta_y f)(x) (\delta_y g)(x) J(dx,dy)  $$
 in $L^2(\mathbb{R}^d,\mu)$. \\
Defining $\mathcal{E}_H^0 (f,g) := \mathcal{E}_{H_G}^0 (f,g) +
\mathcal{E}_{H_J}^0 (f,g)$ with $H_G $ as in (\ref{circle}), (with $\beta(x) = \nabla \ln \phi(x)$), $\mathcal{E}_{H_G}^0$ is the bilinear form
\begin{equation*}
    \mathcal{E}_{H_G}^0 (f,g) = -(H_Gf,g)_{\mu} \:,
\end{equation*}
acting on $f,g \in C_0^2(\mathbb{R}^d)$, in $L^2(\mathbb{R}^d,\mu)$,
and it is a symmetric, positive pre-Dirichlet form in $L^2(\mathbb{R}^d,\mu)$. \\
Since both $\mathcal{E}_{H_G}^0$ and $\mathcal{E}_{H:J}^0$ are symmetric, positive, pre-Dirichlet forms, also
 $\mathcal{E}_H^0$ is a symmetric, positive, pre-Dirichlet form
in $L^2(\mathbb{R}^d,\mu)$, which is closable, under assumptions on
$\phi$ and $\nu$, and the closure is a Dirichlet form in
$L^2(\mathbb{R}^d,\mu)$.\\
\begin{rem}
Following \cite{A-Cufaro} we easily see that $1$ is in the domain of
the closures  $\bar{ H}_G,\bar{H}_J$ and that
$\bar{H_G}\,1=\bar{H}_J=0$ in $L^2(\R^d,\mu)$, thus $\bar{H}\,1=0$
in $L^2(\mathbb{R}^d, \mu)$, it being self-adjoint this is
equivalent with $\bar{H}^\star \,1=0$, hence $\mu$ is invariant
under $e^{-t \bar{H}}$, $t \geq 0$. Hence we have proven the
following theorem :
\end{rem}
\begin{theo}
 Suppose $\phi\,\in\,D({ L}_0^L),\,\phi\,>0\,dx\,a.e.,$ with ${ L}_0^L$ described in (\ref{dddd})  then the operator  $\left(-H, C_0^\infty (\R^d) \right)$ is symmetric in $L^2 (\R^d,\mu) $, with $\mu(\dd x)=\phi^2(x) dx,\,\,x \in\,\R^d),$
 it is also real, hence it has self-adjoint extensions. Under some additional assumptions on $\nu$ and $\phi$, see Remark below, it is essentially self-adjoint on $C_0^\infty(\R^d)$. Its  closure
 $\left( -\bar{H}, D(-\bar{H}) \right)$ is a self-adjoint, non positive definite operator acting in $L^2
 (\R^d,\mu)$. $-\bar{H}$ is the infinitesimal generator of a symmetric Markov process $(Y(t))_{t \in \R_+}$. \,$\mu(\dd x)=\phi^2(x) dx$ is a positive invariant measure for this Markov process.
\end{theo}
\begin{proof}
The analytic statemens have already been proved before. The existence of the symmetric Markov process $(Y(t))_{t \in \R_+}$ generated by $-\bar{H}$ is a result of the theory of Dirichlet forms, see, e.g., \cite{FuOTa}.
\end{proof}
\begin{rem}
$-\bar{H}$ is a L\'evy-type operator in the sense of \cite[pag. 158]{APD} and \cite{Jacob-Sch}. The Markov process generated by $-\bar{H}$ is a
Hunt process by the general theory of Dirichlet form. It solves a
stochastic equation in the weak sense, as a solution of the
associated martingale problem, see \cite{APD}, \cite{Ku}.

\end{rem}
\begin{rem}
We can relate $-H$ to a perturbation $H_V^E$ by a real function V
related to $\phi \in L^2(\R^d)$, called potential and a constant $E
\in  \R$ , of a  symmetric operator $L_0$, defined as $L_{0,G}+L_{0,J}$ acting in $L^2 \left(
\R^d,\,dx\right)$, by
$$
 \left( H_V^E f\right) (x) = (L_0 f) (x) + V^E(x) f(x) \:,\: f \in \left\{C_0^\infty (\R^d) \cup \left\{ c \phi \right\}\; , \; c\in \R \right\} \:\,
$$
with
$$V^E(x) := \frac{[L_0 \phi ](x)}{\phi(x)}+E =
\frac{1}{\phi(x)} \int_{\R^d} \left( \delta_y \phi \right) (x) \nu
(\dd y) +E \:,$$ on $\left\{ x \in \R^d \mid \phi(x) \neq 0
\right\}$, $E$ is a constant such that $ H^E_V \phi = E \phi$.

Under suitable assumptions on $\phi$ and $\nu$ one can prove that
$H_V^E$ is lower semi-bounded and essentially self-adjoint in $L^2
\left( \R^d,\,dx\right),$ its closure denoted by $\bar{H^E_V}$ is
self-adjoint with a spectrum $\sigma
\left(\bar{H^E_V}\right)\,\subset\, [E,\infty)$, and $E$ is an
eigenvalue for $\bar{H^E_V}$.
\end{rem}


\subsection{Certain perturbed O-U L\'{e}vy processes and their invariant measures, via Dirichlet forms}\label{Section2.5}

In this Subsection we start with the finite dimensional case of $\R^d$.
Given a measurable space $(S,\mathcal{B})$, a non-negative valued function $N(x,A)$, $x\in S$, $A \in \mathcal{B}$ is called a kernel on $(S,\mathcal{B})$ if $N(x,\cdot)$ is a positive measure on $\mathcal{B}$ for each fixed $x\in S$ and if $N(x,\cdot )$ is a $\mathcal{B}$-measurable function for each fixed $A\in \mathcal{B}$. If an additional condition
that $N(x,S)\leq 1$, $x\in S$ is imposed, then $N$ is called a Markovian kernel.
We write
$$
    (Nu)(x):=\int_{S} u(y) N(x,\dd y)$$
whenever the integral make sense.
Now let $\mu$ be
 a given $\sigma$-finite Borel measure on $\R^d$. Suppose also that we are given
a kernel $N(x,B)$ on $\R^d \times \mathcal{B}(R^d)$ satisfying the following three conditions:
\begin{enumerate}
    \item  for any $\varepsilon >0$, $N(x,\R^d\setminus U_\epsilon(x))$ is, as function of $x\in \R^d$, locally integrable with respect to $\mu$. Here $U_\epsilon(x))$ is the $\epsilon$-neighbourhood of $x$;
  \item \label{it:N2} $N$ is symmetric, in the sense that
  \begin{align*}
      \int_{\R^d} f(x) (Ng)(x) \mu(\dd x):= \int_{\R^d} (Nf)(x) g(x) \mu(\dd x), \qquad for \ any f,g\in B^+(\R^d),
  \end{align*}
with $B^+(\R^d)$ denoting the set of bounded, Borel measurable mappings on $\R^d$.
\item \label{it:N3} for any compact set $K\subset \R^d$,
\begin{align*}
\int_{\R^d \times \R^d} |x-y|^2 N(x,\dd y) \mu(\dd x) <\infty.
\end{align*}
\end{enumerate}
We notice that condition \ref{it:N2} implies that $N$ determines a positive symmetric Radon measure
$J(\dd x,\dd y)$ on $\R^d \times \R^d \setminus D$ ($D$ is the diagonal set) by
\begin{align*}
   \int_{\R^d \times \R^d \setminus D} f(x,y) J(\dd x, \dd y) = \int_{\R^d \times \R^d}
   f(y,x) J(\dd x, \dd y),
\end{align*}
for any $f\in C_0(\R^d \times \R^d \setminus D)$.

Now put
\begin{align*}
    \mathcal{E}_J (f,g):= \int_{\R^d \times \R^d \setminus D} (f(x)-f(y))(g(x)-g(y)) J(\dd x,\dd y),
\end{align*}
with domain
\begin{align*}
   D(\mathcal{E}_J):= \left\{ f\in L^2(\R^d; \mu): f \ is \ Borel \ measurable, \mathcal{E}_J(f,f) <\infty\right\}.
\end{align*}
Then $\mathcal{E}_J$ is
 a jump
Dirchlet form in the sense of Fukushima
(see \cite[pag. 5]{Fukushima80}) with reference space
$L^2(\R^d; \mu)$.
The proof of the last sentence can be found in \cite[Example 1.2.4., pag. 13]{Fukushima80}).
Moreover, we notice that, due to assumption \ref{it:N3} we know that $C^\infty_0(\R^d)$
is contained in $D(\mathcal{E})$ (see, \cite[pag. 14]{Fukushima80}).

By the general theory on Dirichlet forms to $\mathcal{E}_J$ there is uniquely associated a positive symmetric operator $L_\mu^J$ in $L^2(\R;\dd \mu)$ with domain $D(L_\mu^J) \subset D(\mathcal{E}_J)$.
We are going to exhibit the form of $L_\mu^J$ on $C_0^\infty(\R^d)$. By the relation of $\mathcal{E}_J$ and $L_\mu^J$ we find that
\begin{align}\label{eq:Dirform}
    \mathcal{E}_J(f,g) = \langle - L_\mu^J f,g\rangle,
\end{align}
with $\langle \,,\, \rangle$ the $L^2(\mathbb{R}^d,\mu)-$scalar product and
\begin{align*}
     L_\mu^J f(x)= \int_{\R^d}( f(y)-f(x) )\nu (x;\dd y), \qquad \qquad \mu - a.e.
\end{align*}
and $\nu(x,dy)$ is the Radon-Nykodym derivative of $J(\dd x, \dd y)$ with respect to
$\mu(\dd x)$, that is
\begin{align}\label{eq:nu}
   \int_B \nu(x;\Gamma)\mu(\dd x) = \int_B \, 2 J(\dd x,\Gamma), \qquad \textrm{for any pair of Borel sets } B,\Gamma \textrm{in } \R^d,
\end{align}
provided that this Radon-Nykodym exists.
We notice that by construction,  if $\mu$ is a finite measure then it is infinitesimal invariant for the operator $L^J_\mu$, that is
\begin{align*}
   \int_{\R^d}  L_\mu^J f(x) \dd \mu(x) =0,
\end{align*}
for all $f\in C^\infty_0(\R^d)$.
This follows from the combination of \eqref{eq:Dirform} with the definition of $\nu$ given in \eqref{eq:nu}.
It also follows, $L_\mu^J$ being selfadjoint in $L^2(\R^d,\dd \nu)$, that $\mu$ is
invariant for the semigroup generated by $L_\mu^J$.

We are going to perturb $\mathcal{E}$ by a Dirichlet form $\mathcal{E}_D$ of diffusion type on $C_0^\infty(\R^d)$ mantaining
the Hilbert space $L^2(\R^d;\dd \mu)$. Such kind of forms can be written as
\begin{align*}
      \mathcal{E}^D (f,g):= \frac{1}{2}\int_{\R^d \times \R^d \setminus D} \nabla f(x)Q  \nabla g(x)  \mu (\dd x), \qquad f,g \in D(\mathcal{E}^D),
\end{align*}
with $Q$ a positive and symmetric real valued matrix.
By the general theory such a form is associated with a symmetric positive generator which we call $L^D_\mu$ satisfying the relation
\begin{align*}
   \mathcal{E}^D(f,g)= \langle L_\mu^D f,g\rangle_{L^2(\R^d;\dd \mu)}
\end{align*}
{color{red} Do we have to take into account what follows ? Maybe there is some part which has to be canceled out... }
Assumptions on $\mu$ are known such that $L_\mu^D$ on $C_0^\infty(\R^d)$ takes the form
\begin{align*}
       L^D_\mu f(x)= \frac{1}{2} {\rm Tr}[\sqrt{Q} D^2f(x) \sqrt{Q^*}] + \langle\beta_\mu(x),\nabla f(x)\rangle,
\end{align*}
where $\beta_\mu$ is a vector field in $L^2(\R^d;\dd \mu)$ depending on $\mu$.
Also in this case, if $\mu$ is finite, we easily see that we have infinitesimal invariance
of $\mu$ under $L^D_\mu$ and in fact, invariance, $L^D_\mu$ being symmetric.
Let us consider the sum of the Dirichlet form $\mathcal{E}^J$ and $\mathcal{E}^D$
on $C_0^\infty(\R^d)$ in $L^2(\R^d;\dd \mu)$. With the previous assumptions on $J$ and $\mu$ the closure of this sum is still a Dirichlet form $\mathcal{E}$ with domain
$D(\mathcal{E})$ exists in $L^2(\R^d;\dd \mu)\times L^2(\R^d;\mu)$. Let us call
$L$ the selfadjoint associated operator in $L^2(\R^d;\dd \mu)$. Then
\begin{align*}
    \mathcal{E}(f,g)= \langle L f,g \rangle_{L^2(\R^d;\dd \mu)}
\end{align*}
with
\begin{align*}
    Lf(x)&=  L^D_\mu f(x)+ L_\mu^J f(x)\\
         &= \frac{1}{2} {\rm Tr}[\sqrt{Q} D^2f(x) \sqrt{Q^*}] + \langle\beta_\mu(x),\nabla f(x)\rangle + \int_{\R^d\setminus\left\{x\right\}}( f(y)-f(x) )\nu (x;\dd y), \qquad f\in C^\infty_0(\R^d),
\end{align*}
where $L_\mu^J,L^D$ are the generators considered above one has \eqref{eq:nu} for $\nu$.
We notice that the integral part in the expression of $L$ can be rewritten as
\begin{align*}
       \int_{\R^d\setminus\left\{0\right\}}( f(x+y)-f(x) )\nu (x;x+\dd y);
\end{align*}
with this change the operator $L$ becomes a particular case of the form considered in \eqref{eq:op-non-loc}.
Moreover, if $\mu$ is finite,
then $\mu$ is infinitesimal invariant under $L$ and invariant under the generated semigroup
$P_t:=e^{tL}, t\geq 0$ in $L^2(\R^d;\dd \mu)$. By the general theory of regular Dirichlet forms there is a Hunt process $(X(t))_{t\geq 0}$ in $\R^d$ properly associated with $\mathcal{E}$, whose transition semigroup is $(P_t)_{t\geq 0}$, i.e.
\begin{align*}
    (P_t f)(x) = \mathbb{E}[f(X(t))].
\end{align*}
In the following we exhibit the stochastic differential equation satisfied by the process
$(X(t))_{t\geq 0}$. To this end we recall that for our infinitesimal  generator
\begin{align*}
       Lf(x)
         = \frac{1}{2} {\rm Tr}[\sqrt{Q} D^2f(x) \sqrt{Q^*}] &+ \langle\beta_\mu(x),\nabla f(x)\rangle \\ &+ \int_{\R^d\setminus\left\{0\right\}}( f(x+y)-f(x) )\nu (x;x+\dd y)\;,\; f\in C^\infty_0(\R^d),
\end{align*}
if $\nu(x,x+dy)$ has a Radon-Nikodym density $ \zeta(x,x+y)$ with respect to some positive measure $\tilde{\nu}$ on $\mathcal{B}(\mathbb{R}^d)$ , and then $\nu(x, x+ \Gamma)=\int\limits_\Gamma  \zeta(x,x+y) \tilde{\nu} (dy)$ holds.
The associated stochastic integral equation is
\begin{equation}
\begin{aligned}\label{eq:int-eq}
   X(t)&= X(0) + \int_0^t \sqrt{Q} \dd B(s) + \int_0^t \beta_\mu(X(s))\dd s \\
      & + \int_0^t \int_{|y|<1} \zeta(X(t),y)y \tilde{N}(\dd s, \dd y) +
     \int \int_{|y|\geq 1} \zeta(X(t),y)y N(\dd s, \dd y),
\end{aligned}
\end{equation}
where $(B(t))_{t \geq 0}$ is a standard $d$-dimensional Brownian motion,
$N(\dd s, \dd y)$ is a Poisson random measure (independent of $(B(t))_{t\geq 0}$)
associated with a point process on $\R^d$ with intensity measure
$\tilde{\nu}$, such that $\tilde{N}(\dd s,\dd y)$ is the compensated Poisson random measure, i.e.
\begin{align*}
     \tilde{N}([0,t],\dd y) = N([0,t], \dd y) - t \nu (\dd y)\,.
\end{align*}
One has $\tilde{\nu}(U)= \mathbb{E}[N([0,1],U)], U \in \mathcal{B}(\R^d)$ and
$\zeta$ is such that
\begin{align*}
   \nu(x,x+\Gamma) = \int_{\Gamma} \zeta(x,x+y) \nu(\dd y),
\end{align*}
(see \cite{ImkWil12} for more detail on the definition of $\zeta$).
Taking into account the arguments above, in particular \eqref{eq:nu}, the relation between $\zeta$, $J$, $\mu$ and $\nu$ can be expressed as follows:
\begin{align*}
     \int_B \int_{\Gamma} \zeta(x,x+y) \nu(\dd y) \mu(\dd x) =
       \int_B \int_{\Gamma} 2 J(\dd x, x+\dd y), \qquad for \ any \ B,\Gamma \in \mathcal{B}(\R^d).
\end{align*}
This shows that $\zeta$ also is the Radon-Nikodym derivative of $J$ with respect the product measure $\mu \times \nu$ on $\R^d \times \R^d$.
\begin{rem}
Arguing as in \cite{PeZa}
the integral equation \eqref{eq:int-eq} can also be written in differential form as
\begin{equation}\label{HAJOM}
    \dd X(t) = \sqrt{Q} \dd B(t) + \beta_\mu(X(t))\dd t + G(X(t)) \dd L(t)
\end{equation}
where $(L(t))_{t\geq 0}$ is a L\'evy process with values in the
space $U:=\mathcal{M}(\R^d)$, with $ \mathcal{M}(\R^d)$ the space
of $\sigma-$finite signed measures on $\R^d$, and for any $x\in
\R^d$, $G(x): U \to \R^d$ is the linear map given by
\begin{equation}\label{hajom}
     G(x)\lambda = \int_{\R^d} \zeta(x,x+y) y \lambda(\dd
     y),\,\,\,\lambda \in \mathcal{M}(\R^d)\;,
\end{equation}
the integral in (\ref{hajom}) being assumed to exists.\\
 The finite dimensional distributions of $(L(t))_{t\geq 0}$ coincide with those given by
\begin{align*}
   \int_0^t \int_{|y|<1}y \tilde{N}(\dd s, \dd y) +
     \int \int_{|y|\geq 1} y N(\dd s, \dd y).
\end{align*}
We note that the representation (\ref{HAJOM}) can be put in relation
with (\ref{SDEsymb}), see, eg., \cite{BRHA}.
\end{rem}



\section{Invariant measures in infinite dimensions} \label{Section3}
\subsection{The case of  the infinite dimensional O-U L\'evy process}\label{Section3.1}
We shall work in the setting of \cite{ADPMS1}. We consider the
linear stochastic differential equation:

\begin{equation}\label{eq:OU}
\begin{aligned}
   &\dd X(t)= A X(t) \dd t + \dd L(t), \qquad t\geq 0,\\
   & X(0)=x \in \cH,
\end{aligned}
\end{equation}
where $\cH$ is a real separable Hilbert space, $(L(t))_{t\geq 0}$ is
an infinite dimensional cylindrical symmetric L\'evy process and $A$
is a self-adjoint operator generating a  $C_0$-semigroup in $\cH$.\\
We further assume that $A$ is strictly negative such that there exists a
basis $({e_n})_{n\,\in\,\N}$ in $\cH$  verifying
\begin{equation}
({e_n})_{n\,\in\,\N}\,\,\subset\,\,D(A),\,\,\,\,\,\,\,\,A\,e_n=-\lambda_n\,e_n,
\end{equation}
where
$\lambda_n\,>0,\,n\,\in\,\N_0,\,\,\,\,\lambda_n\,\uparrow\,+\infty.$\\
Assume moreover that for some $\beta_n>0,\,n\,\in\,\N,$ we have
\begin{equation}
\sum_{n=1}^\infty \left( \beta^2_n \int_{|y|<1/\beta_n} y^2
\nu_\mathbb{R}(\dd y) + \int_{|y|\geq 1/\beta_n} \nu_\mathbb{R}(\dd
y)\right)<+\infty,
\end{equation}
 for
some symmetric L\'evy measure $\nu_{\R}$ on $\R$, (i.e.
$\nu_{\R}(-A)=\nu_{\R}(A),\,\,\forall\,A\,\in\,\cB(\R)).$ We set
\begin{equation}
L(t)=\sum_{n=1}^\infty \beta_n L^n(t) e_n,
\end{equation}
with $L^n(t)$ defined by
\begin{equation}
 \mathbb{E}[e^{ihL^n(t)}] = e^{-t \psi_\R(h)}, \quad h\in
 \R,\,t\,\geq\,0,
\end{equation}
and
\begin{equation}
 \psi_\R(h)= \int_{\R} (1-\cos(hy) \nu_\R(\dd y), \quad h\in \R.
\end{equation}
As shown in \cite{ADPMS1} if
\begin{equation}
\int_1^{+\infty} \log(y)\nu_\R(\dd y) <\infty,
\end{equation}
and
\begin{equation}
\sum_{n=1}^\infty \frac{1}{\lambda_n} <\infty,
\end{equation}
then the L\'evy driven Ornstein-Uhlenbeck process
$X=(X(t))_{t\geq\,0}$ given by
\begin{equation}
X(t)=e^{tA}x+\sum_{n=1}^\infty \left( \int_0^t
   e^{-\lambda_n(t-s)} \beta_n \dd L^n(s)\right) e_n \notag,
\end{equation}
is well defined, in the sense that the series is convergent in the probability sense, and the
process
$X$ is adapted, i.e. $X(t)$ is $\cF_{t}$-measurable and Markovian, see.e.g. \cite[Th.2.8]{PrZa}.\\
$X(t)$ solves $dX(t)=A\,X(t)\,dt+dL(t)$ in the mild sense. It is
shown in [\cite{ADPMS1}, Proposition 2.5] that $X$ admits a unique
invariant probability measure (i.e. X(t) is invariant under the
Markovian transition semigroup associated to $X(t).$
\begin{rem}
The existence of an invariant measure has also been proven in
another non necessary cylindrical stting with related conditions in
\cite{ChoMi}.
\end{rem}

Since the semigroup $e^{-tA}$ is stable in $H$ (we recall that $A$
is strictly negative), we can apply Theorem 3.3 in \cite{ChoMi}, and
we get that the invariant measure $\mu$ for $X(t)$ is of the form
$\mu=\nu_G\,\ast\,\nu_J$, where
\begin{equation}\nu_G(dx)=\,N(0;\,A^{-1})(dx),\,x\,\in\,\cH,\end{equation}
and
\begin{equation}
\nu_J(B)={\cal{L}}\,\Big( \dint_0^\infty\,e^{-sA}\,dL(s) \Big)(B), B
\in\, \cal{B} (\cal{H}).
\end{equation}
\begin{enumerate}
\item Note that \cite{ChoMi} proved in particular that
$\dint_0^\infty\,e^{-sA}\,dL(s)$ exists as an infinitely divisible
distribution with L\'evy characteristics $$\Big(0, \dint_0^t
\nu(\gamma_s^{-1} x)ds, \dint_0^\infty [\chi_B(\gamma,
x)-\chi_B(x)]\,\nu(dx) ds \Big).$$
\item Note also that this representation is completely analogous to
the one in finite dimensions, see, [\cite{Sato}, Lemma 17.1].
\end{enumerate}
We remark that both $\nu_G$ and $\nu_J$ are weak limits of their
restrictions $\nu_G^{(n)},\,\nu_J^{(n)}$ onto the finite dimensional
subspaces spanned by the ${\{e_1,...,e_n\}}$ in $\cH$.

\subsection{Certain perturbed  infinite dimensional O-U L\'evy processes } \label{Section3.3}
Let us now indicate how to extend the approach developed in previous sections in the finite dimensional setting to the case where $\mathbb{R}^d$ is
replaced by a separable Hilbert space $\mathcal{H}$.

The theory of Dirichlet forms on such spaces is well developed, see
\cite{Alb, AMR, MaRo}, and references therein. Let $\mu$ be a
probability measure on $\mathcal{H}$ . We assume that $\mu$ is
admissible in the sense of \cite{MaRo}. Let $\mathbcal{E}_\mu^D$ be a
classical, quasi regular, Dirichlet form (in the sense of
\cite{AlRo, MaRo})  acting on $D\left(
\mathbcal{E}_\mu^D\right) \subset L^2(\mathcal{H},\mu)$. To it there
is uniquely associated a self-adjoint operator $L_\mu^D$ with domain
$D(L^D_\mu)$ acting in $L^2(\mathcal{H},\mu)$ such that $-L_\mu^D
\geq 0$ and $\mathbcal{E}_\mu^D(f,g)= \left( f, (-L_\mu^D) g
\right)_{L^2(\mathcal{H},\mu)}\;,$ for all $ f \in
D(\mathbcal{E}_\mu^D)$, $g \in D\left( L_\mu^D\right)$.

Let $FC^\infty_b$ the family of cylinder functions which are $C^\infty$ and with bounded derivatives of any orders on the basis. By the definition of quasi regular Dirichlet
 forms $FC_b^\infty$ is dense in $L^2(\mathcal{H},\mu)$. We have that $\left(-L^D_\mu g \right) (x) = \Delta g +\beta_\mu \cdot \nabla g \;,$ with $\beta_\mu \in L^2(\mathcal{H},\mu)$ and
  $\Delta g$, $\nabla g$ defined in the natural way, see \cite{MaRo}.

As in the finite dimensional case we have that $\mu$ is invariant under the semigroup $e^{t L_\mu^D}$.

Let us consider a symmetric , Borel measure on $\left(\mathcal{H} \times \mathcal{H}\right) \setminus D$, where $D$ is the diagonal in $\mathcal{H} \times \mathcal{H}$, and consider the associated jump Dirichlet form $$\mathbcal{E}_\mu^J (f,g)= \int_{\mathcal{H}} \int_{\mathcal{H}} [f(x)-f(y)] [g(x)-g(y)] J(dx,dy)\;, f,g \in D\left( \mathbcal{E}_\mu^J \right) \subset L^2(\mathcal{H},\mu)\:.$$
Under some conditions on $\mu$ and $J$, we have that $\mathbcal{E}_\mu^J$ exists,  as the closure of its restriction to $f,g \in FC^\infty_b$ in $L^2(\mathcal{H},\mu)$, see \cite{ASK}. The corresponding self-adjoint operator $L_\mu^J$ has the form
$$
 \left( L_\mu^Jf\right)(x) = \int [f(y)-f(x)] \nu^{J,\mu}(x,dy) \:,
$$
provided $J(dx,dy)$ is absolutely continuous with respect to
$\mu(dx)$. We denoted by $\nu^{J,\mu}(x,dy)=\frac{2
J(dx,dy)}{\mu(dx)}$ the corresponding Radon-Nikodym derivative
(multiplied by $2$).

As in the finite dimensional case, we have that $\mu$ is invariant
under the semigroup $e^{t L_\mu^J}$, $t \geq 0 $, generated by
$L_\mu^J$. By the general theory, see, e.g., \cite{MaRo},
$\mathbcal{E}_\mu= \mathbcal{E}^D_\mu + \mathbcal{E}^J_\mu$ is a
Dirichlet form on $L^2(\mathcal{H},\mu)$, with an associated
self-adjoint operator $L_\mu$ such that $L_\mu= L^D_\mu +L^J_\mu$,
on $D(L_\mu^D) \cap D(L_\mu^J) \supset FC^\infty_b$, in
$L^2(\mathcal{H},\mu)$. Moreover $\mu$ is invariant under the $C_0$
Markov semigroup $e^{t L_\mu}$, $t \geq 0$, generated by $L_\mu$.

By the general theory, see \cite{AMR}, there is a  decomposition for
the Markov process $X_t$ properly associated with
$\mathbcal{E}_\mu$. For any $f \in D\left(\mathbcal{E}_\mu\right)$
we have
\begin{equation}\label{OM1}
  f(X_t) = f[X_0]+ N_t^{[f]}+ M_t^{[f]} \,, \mathbb{P}^\mu \; a.s. \;,
\end{equation}
where $N_t^{[f]}$ is a smooth zero-energy additive functional, and
$M_t^{[f]} $ is an additive martingale functional. So far for the
general theory on $\mathcal{H}$. Let us now briefly indicate how to
relate such structures to the corresponding finite dimensional ones
discussed in chapter 2.\\ Let us first take $\mu$ to be the
invariant measure of the $O-U$ process on $\mathcal{H}$ perturbed by
a non linear drift term which we discussed in \cite{ADPM}. In
particular $\mu$ has the form, $e^{-G} \frac{\mu_A}{\int_\mathcal{H}
e^{-G} d\mu_A}  $ where $G$ is such that $G^\prime = F$ in the
Fr\'echet sense, $\mu_A$ is the Gaussian probability measure which
is invariant for the O-U process with linear drift $A$, i.e. $\mu_A
= N\left(0,A^{-1}\right)$. Then $\mu$ is the invariant measure of
the process solving
$$
 dX_t = \left[ A X_t + F(X_t) \right] dt +  dW_t \;,
$$
with $A,F$ and $W$ as in \cite{ADPM}.\\
In this case we have thus, in particular, that the linear function
is in $D\left(\mathbcal{E}_\mu\right)$ and (\ref{OM1}) holds, with
$$N_t=W_t,\,\,\,\,\,\,M_t=\dint_0^t F(X_s)\,ds$$

In the construction of $\mu$ in \cite{ADPM} we used finite
dimensional approximations, together with the cylindrical structure
of $W_t,$ hence the relation with Chapter 2 is established in this
case of a Gaussian additive noise. \\In the case where $\mu$ is
not absolutely continuous with respect to some reference Gaussian measure,
one has to go through a more involved analysis. Elements of it have been
already indicated in \cite{AlRo}. We plan to carry out this programme in further
publications.

\newpage

\section*{Acknowledgments}
\noindent  This work was supported by King Fahd University of
Petroleum and Minerals under the project $\sharp\,IN121060$. The
authors gratefully acknowledge this support.\\We thank Stefano
Bonaccorsi and
Luciano Tubaro at the University of Trento for many stimulating discussions.\\
The authors would also like to gratefully acknowledge the great
hospitality of various institutions. In particular, for the first
author, CIRM, the Mathematics Department of the University of Trento,
and the Department of Computer Science of the University of Verona;
for him and the fourth author, King Fahd University of Petroleum and
Minerals at Dhahran; for the second, third  and fourth authors IAM
and HCM at the University of Bonn, Germany.

\medskip

\begin{flushleft}
\footnotesize
{\it S. Albeverio}  \\
 Dept. Appl. Mathematics, University of Bonn,\\
HCM;  BiBoS, IZKS, KFUPM(Dhahran); CERFIM (Locarno)\\
\medskip
{\it L. Di Persio}\\
University of Verona, Department of Computer Science, \\
 Strada le Grazie 15 - 37134 Verona,  Italia \\
\medskip
{\it E. Mastrogiacomo}\\
Universit\'a degli studi di Milano Bicocca, Dipartimento di Statistica e Metodi Quantitativi \\
   Piazza Ateneo Nuovo, 1 20126 Milano
\\
\medskip
{\it B. Smii} \\
Dept. Mathematics, King Fahd University of Petroleum and Minerals, \\
Dhahran 31261, Saudi Arabia\\

\medskip
 E-mail: albeverio@uni-bonn.de \\
\hspace{1,1cm} luca.dipersio@univr.it\\
\hspace{1,1cm} elisa.mastrogiacomo@polimi.it\\
\hspace{1,1cm}boubaker@kfupm.edu.sa

 \end{flushleft}
\end{document}